\documentclass[reqno,english,11pt]{amsart}
\usepackage{times}

\usepackage{amsmath,amsfonts,amssymb,graphicx,amsthm,enumerate,url,mathabx}
\usepackage[noadjust]{cite}
\usepackage{stmaryrd}
\usepackage{mathrsfs,booktabs,tabularx,setspace}
\usepackage{xifthen,xcolor,tikz}
\usetikzlibrary{arrows}
\usetikzlibrary{decorations.pathmorphing,patterns,shapes,calc,decorations}
\usetikzlibrary{decorations.pathreplacing}
\usepackage{mathtools,upgreek}
\usepackage[final]{showkeys} 

\definecolor{refkey}{gray}{.75}
\definecolor{labelkey}{gray}{.5}
\usepackage[algo2e,boxed,vlined,algoruled]{algorithm2e}

\usepackage[letterpaper]{geometry}
\usepackage[colorinlistoftodos]{todonotes}
\presetkeys{todonotes}{inline, color=green}{}

\usepackage{accents}
\usepackage{float}

\usepackage[colorlinks=true]{hyperref}
\colorlet{DarkGreen}{green!50!black}
\colorlet{DarkGray}{gray!60!black}

\usepackage{subcaption}

\numberwithin{equation}{section}

\newcommand{\ignore}[1]{}

\renewcommand{\epsilon}{\varepsilon}

\newcommand{\one}{\mathbf{1}}
\newcommand{\biota}{\boldsymbol{\iota}}
\newcommand{\zero}{\mathbf{0}}

\usepackage{color}
 \definecolor{refkey}{gray}{.5}
 \definecolor{labelkey}{gray}{.5}
\definecolor{light}{gray}{.9}

\newtheorem{theorem}{Theorem}[section]
\newtheorem*{theorem*}{Theorem}
\newtheorem{lemma}{Lemma}[section]

\newtheorem{proposition}{Proposition}[section]

\newtheorem{corollary}{Corollary}[section]

\theoremstyle{definition}{

\newtheorem{definition}[theorem]{Definition}

\newtheorem*{definition*}{Definition}

\newtheorem{remark}{Remark}[section]

}

\newcommand{\bbT}{\mathbb T}

\newcommand{\cC}{\ensuremath{\mathcal C}}

\newcommand{\cE}{\ensuremath{\mathcal E}}
\newcommand{\cF}{\ensuremath{\mathcal F}}

\newcommand{\out}{{\textrm{out}}}
\newcommand{\Ext}{{\mathsf{Ext}}}
\newcommand{\Int}{{\mathsf{Int}}}

\newcommand{\bomega}{{\boldsymbol{\omega}}}
\newcommand{\bgamma}{{\boldsymbol{\gamma}}}
\newcommand{\Comp}{{\rm{Comp}}}

\newcommand{\ord}{{\ensuremath{\mathsf{ord}}}}

\newcommand{\dis}{{\ensuremath{\mathsf{dis}}}}

\newcommand{\tmix}{t_{\textsc{mix}}}

 \renewcommand{\epsilon}{\varepsilon}

\DeclareMathOperator{\diam}{diam}

\newcommand{\ext}{{\mathsf{ext}}}

\newcommand{\tv}{{\textsc{tv}}}

\makeatletter
\newcommand{\superimpose}[2]{%
  {\ooalign{$#1\@firstoftwo#2$\cr\hfil$#1\@secondoftwo#2$\hfil\cr}}}
  
\newcommand{\sbullet}{%
  \hbox{\fontfamily{lmr}\fontsize{.4\dimexpr(\f@size pt)}{0}\selectfont\textbullet}}

\makeatother

\begin{document}

\title{Spatial mixing and the random-cluster dynamics on lattices}

\author[Reza Gheissari]{Reza Gheissari}
\address{R.\ Gheissari\hfill\break
 Departments of Mathematics \\ Northwestern University }
 \email{gheissari@northwestern.edu}

\author[Alistair Sinclair]{Alistair Sinclair}
 \address{A.\ Sinclair\hfill\break
 Computer Science Division \\ UC Berkeley }
 \email{sinclair@berkeley.edu}

\maketitle
\thispagestyle{empty}

\vspace{-.5cm}
\begin{abstract}
    An important paradigm in the understanding of mixing times of Glauber dynamics for 
        spin systems is the correspondence between spatial mixing properties of the models and bounds on the mixing time of the dynamics. This includes, in particular, the classical notions of weak and strong spatial mixing, which have been used to show the best known mixing time bounds in the high-temperature regime for the Glauber dynamics for the Ising and Potts models.
    
    Glauber dynamics for the random-cluster model does not naturally fit into this spin systems framework because its transition rules are not local. In this paper, we present various implications between weak spatial mixing, strong spatial mixing, and the newer notion of spatial mixing within a phase, and mixing time bounds for the random-cluster dynamics in finite subsets of $\mathbb Z^d$ for general $d\ge 2$. These imply a host of new results, including optimal $O(N\log N)$ mixing for the random cluster dynamics on torii and boxes on $N$ vertices in $\mathbb Z^d$ at all high temperatures and at sufficiently low temperatures, and for large values of $q$ quasi-polynomial (or quasi-linear when $d=2$) mixing time bounds from random phase  initializations on torii at the critical point (where by contrast the mixing time from worst-case initializations is exponentially large). In the same parameter regimes, these results translate to fast sampling algorithms for the Potts model on $\mathbb Z^d$ for general~$d$.
\end{abstract}


\section{Introduction}
Let $G=(V,E)$ be a finite graph.  Configurations of the {\it random-cluster model\/}
on~$G$ are sets of edges $\omega\subseteq E$, with an associated probability given
by the {\it Gibbs distribution\/}:
\begin{equation}\label{eqn:RCgibbs}
   \pi_{G,p,q}(\omega) = \frac{1}{Z_{G,p,q}}\, p^{|\omega|} (1-p)^{|E| - |\omega|} q^{|\rm{Comp}(\omega)|},
\end{equation}
where $\rm{Comp}(\omega)$ is the set of connected components in the subgraph $(V,\omega)$ and the normalizing factor $Z_{G,p,q}$ is known as
the {\it partition function}.  Here $p\in[0,1]$ and $q\in[1,\infty)$ are 
parameters\footnote{The random-cluster model is
actually defined for all positive real~$q$; the case of $q\in (0,1)$ is interesting in its own right, but is phenomenologically very different from $q\ge 1$ and thus beyond the scope of the current paper.}.
Note that the distribution~\eqref{eqn:RCgibbs} generalizes the standard 
Erd\"os-R\'enyi random subgraph model (or edge percolation model) on~$G$, with the inclusion of the extra factor~$q^{|\rm{Comp}(\omega)|}$. This factor assigns
higher weight to configurations with more connected components, and this bias
increases with~$q$.

The random-cluster model was introduced in the late 1960s by Fortuin and
Kasteleyn~\cite{FK} as a unifying framework for studying percolation, spin
systems in statistical physics and random spanning trees; see the 
book~\cite{Grimmett} for extensive background.  The connection with spin
systems is via the {\it $q$-state Potts model}, whose configurations are
assignments $\sigma\in\{1,\ldots,q\}^V$ of one of $q$ possible spins, or colors, to each
vertex of~$G$, with associated Gibbs distribution.
\begin{equation}\label{eqn:Pgibbs}
   \pi_{G,p,q}^{\rm Potts}(\sigma) = \frac{1}{Z^{\rm Potts}_{G,\beta,q}}\,\exp(-\beta|{\rm Cut}(\sigma)|),
\end{equation}
where $\beta\ge 0$ is a parameter, ${\rm Cut}(\sigma)$ is the set of edges connecting different spins (i.e., 
the edges in the multi-way cut induced by~$\sigma$), and ${Z^{\rm Potts}_{G,\beta,q}}$ is the partition function.  The case of two spins
($q=2$) is the Ising model.  For integer $q$, under the correspondence
$p=1-\exp(-\beta)$, the random-cluster and Potts models are intimately related,
in the sense that both are marginals of the same joint distribution on
edges and spins (the so-called {\it Fortuin-Kasteleyn (FK)  distribution\/}~\cite{ES}).  Specifically, if~$\sigma$
is chosen from the Potts distribution~\eqref{eqn:Pgibbs}, and we include in $\omega$ the edges connecting equal spins in $\sigma$ independently with probability 
$p = 1-\exp(-\beta)$, then $\omega$ is distributed according
to the random-cluster distribution~\eqref{eqn:RCgibbs}; and conversely, if we pick~$\omega$ from~\eqref{eqn:RCgibbs} and assign spins $\{1,\ldots,q\}$  independently and u.a.r.\ to each
connected component in~$(V,\omega)$, then the resulting spin configuration is distributed according to~\eqref{eqn:Pgibbs}.
However, the random-cluster model is more general in that, unlike the Potts model, it is defined even when $q$ is not an integer.

In this paper we focus for definiteness on the classical setting where $G=\Lambda_n$
is a box of
side-length~$n$ in the $d$-dimensional lattice~${\mathbb Z}^d$.  In this setting,
the random-cluster model undergoes a {\it phase transition\/} at a certain
{\it critical\/} value~$p=p_c(q,d)$: namely, as $n\to\infty$, if $p>p_c$ then
the configuration~$\omega$ has w.h.p.\ a giant component (of size linear in
the number of vertices $N=n^d$), while if $p<p_c$ the largest component is 
w.h.p.\ of size only $O(\log N)$~\cite{DCRT19}.\footnote{For $q=1$, this is just the well-known
phenomenon of emergence of a giant component in Bernoulli percolation.}  For integer $q\ge 2$,
an analogous phase transition occurs in the Potts model: for 
$\beta>\beta_c(q,d)=-\ln(1-p_c(q,d))$, the Gibbs distribution~\eqref{eqn:Pgibbs} 
becomes multi-modal, with each of the $q$ modes corresponding to a {\it phase\/}
in which the configuration is dominated by one of the spins, while for $\beta<\beta_c$
correlations between spins decay exponentially with distance and there is no
long-range order.  The regimes $\beta<\beta_c$ and $\beta>\beta_c$ are referred
to as the {\it high-temperature\/} and {\it low-temperature\/} regimes, respectively,
because of the interpretation of~$\beta$ in statistical physics as an ``inverse temperature" parameter.  With slight abuse of terminology, we shall use the same phrases
to refer to the analogous regimes $p<p_c$ and $p>p_c$ in the random-cluster
model.  For a more precise formulation of this phase transition, see 
Section~\ref{sec:preliminaries}.

We study the natural local dynamics for the random-cluster 
model, analogous to the widely studied Glauber dynamics for the Potts model
that updates a random spin at each step conditional on the spins of its
neighbors.
This random-cluster dynamics, which we abbreviate to {\it FK dynamics}, 
at each step picks a random
edge $e\in E$ and includes~$e$ in the current configuration~$\omega$ with
the correct conditional probability given~$\omega\setminus e$.  
Thus, if the configuration at time $t$ is~$\omega$, then the
configuration at time $t+1$ is obtained as follows:
\begin{enumerate}
\item Pick an edge $e\in E$ uniformly at random
\item Include $e$ in the new configuration with probability  
\begin{align}\label{eq:update-rule}
   \begin{cases}
      \frac{p}{q(1-p)+p} & \hbox{\rm if $e$ is a \emph{bridge} in~$\omega\cup \{e\}$ (i.e., if $|{\rm Comp}(\omega \setminus\{e\})| \ne |{\rm Comp}(\omega\cup \{e\})|$)};\\
      p & \hbox{\rm otherwise}.
   \end{cases}
\end{align}
\end{enumerate}
Since this dynamics is irreducible and reversible w.r.t.\ the Gibbs distribution~$\pi$ in~\eqref{eqn:RCgibbs}, it converges to this distribution from any
initial configuration.  As in all such investigations, the main question is
the {\it mixing time\/} of the chain, i.e., the number of steps until 
it gets close in total variation distance to~$\pi$ from a worst-case initialization.  This is relevant both for bounding the running time
of a MCMC algorithm that samples configurations from~$\pi$, and for understanding
the actual behavior of the Markov chain over time.

The analysis of Glauber dynamics has a long and rich history, and the
dynamics of the Ising model (and to a lesser extent the Potts model) are by now quite well understood. By comparison, the study of the FK dynamics is less well developed and introduces new technical challenges, mainly due to 
the fact that the update rule~\eqref{eq:update-rule} necessarily depends on the
number of connected components, $|\rm{Comp}(\omega)|$, a {\it non-local\/} quantity.
Many of the most important results for the Ising/Potts Glauber dynamics on $\mathbb Z^d$ relate the mixing time
to some form of {\it spatial mixing\/} (i.e., decay of correlations with distance
in the Gibbs distribution).  Our goal in this paper is to overcome the non-locality and develop analogous implications for the FK dynamics.  In the
following paragraphs we explain each of these implications, along with their
consequences for the mixing time of the FK dynamics and sampling from the Potts model in various parameter regimes.

\subsection{Weak spatial mixing}\label{subsec:introwsm}
The most basic notion of spatial mixing is \emph{weak spatial mixing (WSM)}, which
expresses the fact that in $\mathbb Z^d$, or on a torus $\bbT_n = (\mathbb Z/n\mathbb Z)^d$, when the status of the edges in some subset~$A$ of the
configuration are changed, the effect on the distribution of the 
edges in some other set~$B$ decays exponentially with the distance between~$A$
and~$B$. (For a precise formulation, see Definition~\ref{def:WSM}.)

The following theorem says that, whenever the random-cluster model has WSM, the FK dynamics on $\bbT_n$ has very fast (actually, optimal) mixing time.  In general, WSM is believed to be the 
weakest condition under which one can hope to have optimal mixing~\cite{DSVW}.
\begin{theorem}\label{thm:intro1}
Suppose $d\ge 2, q\ge 1, p$ are such that WSM holds for the random-cluster model. Then the mixing time of the FK dynamics on $\bbT_n = (\mathbb Z/n\mathbb Z)^d$ is $O(N\log N)$, where $N= n^d$ is the number of vertices. 
\end{theorem}

This implication is the analog of a similar one for Glauber dynamics for
the Ising model (the $q=2$ case of the Potts model), which is
a landmark result in the field~\cite{MaOl1}.  Indeed, Harel
and Spinka~\cite{HarelSpinka} recently observed, towards a different purpose, 
that the proof technique of~\cite{MaOl1}
in fact extends to more general monotone models than the Ising model, including the
random-cluster model.\footnote{The results of \cite{MaOl1,HarelSpinka} are
actually formulated for the infinite lattice~${\mathbb Z}^d$ rather than
for the torus~$\bbT_n$.} Thus Theorem~\ref{thm:intro1} can be deduced fairly
easily from the arguments in~\cite{HarelSpinka},
but we include it here because the important implication for mixing times in finite volumes is not explicitly derived there.

When $d=2$, WSM is known to hold for the random-cluster model at all non-critical~$p$ for all~$q$ from~\cite{BD1} together with planar duality.
For general $d\ge 3$, WSM was recently shown throughout the high-temperature regime $p<p_c(q,d)$ for all $q$~\cite{DCRT19},
and also throughout the low-temperature regime $p>p_c(2,d)$ for the special case
$q=2$~\cite{DCGR20}. WSM can also easily be seen to hold for all $q$ and $d$ provided the temperature is low enough, i.e., $p> p_0(q,d)$ is 
sufficiently large (see Section~\ref{subsec:wsm}).
In fact, WSM is conjectured to hold at all non-critical temperatures $p\ne p_c(q,d)$ for all $q$ and~$d$~\cite[Section 1.4]{DCGR20}.

Combining these facts with Theorem~\ref{thm:intro1}, we immediately obtain the following.
\begin{corollary}\label{cor:intro1}
For all $d\ge 2$, the FK dynamics on $\bbT_n$ mixes in $O(N \log N)$ time
in the following cases:
\begin{itemize}
\item[(i)] Throughout the high-temperature regime $p<p_c(q,d)$ for all $q\ge 1$.
\item [(ii)] Throughout the low-temperature regime $p>p_c(2,d)$ for $q=2$.
\item [(iii)] At all sufficiently low temperatures $p>p_0(q,d)$ for all  $q\ge 1$.
\end{itemize}
\end{corollary}
\noindent
Moreover, as mentioned above, if as conjectured WSM holds at all $p\ne p_c(q,d)$,
then Theorem~\ref{thm:intro1} will immediately imply optimal mixing
time on $\bbT_n$ for all non-critical~$p$ for all $q$ and $d$.

Corollary~\ref{cor:intro1} has interesting implications for sampling
algorithms, not only for the random-cluster model but also, in light
of the intimate connection with the Potts model, for that as well: see Section~\ref{subsubsec:sampling}.

\subsection{Strong spatial mixing}\label{subsec:introssm}
Another widely used notion of spatial mixing is known as \emph{strong spatial mixing
(SSM)}; this notion is more adapted to finite domains such as~$\Lambda_n$, 
which we recall is a box of side-length~$n$ in~$\mathbb Z^d$. SSM captures the same correlation decay property between two subsets~$A$ and~$B$ as WSM but in the context of a fixed environment, called a \emph{boundary condition}, which we can understand to be a fixed configuration of edges (or spins)~$\xi$ on the exterior of the region~$\Lambda_n$; we use the notation~$\Lambda_n^\xi$ to denote the presence of boundary condition~$\xi$. (See Section~\ref{subsubsec:bc} for a precise formulation.) Notice that SSM is in general a stronger notion than WSM since the fixed boundary edges~$\xi$ may be very close to the subset~$B$ (much closer than the edges~$A$ whose status is being varied), which may have the effect of increasing the correlation between $A$ and $B$. We emphasize that this discussion includes the case of a \emph{free} boundary (i.e., no fixed edges/spins) which already causes the results of the previous subsection to break down because the transitivity present in~$\bbT_n$ is lost. Indeed any treatment of dynamics in the finite region $\Lambda_n$ must address the boundary condition.

In the classical literature concerning the Ising and Potts Glauber dynamics, the SSM property \emph{uniformly over all possible boundary conditions} is the standard criterion that implies rapid mixing on~$\Lambda_n$, see e.g.,~\cite{MaOl1,MaOl2,Cesi}.
Such results generally require uniformity of SSM over boundary conditions because
they employ a recursive argument in which one loses control of the specific
boundary conditions present at subsequent levels of the recursion.
With the exception of the Ising model~\cite{DSS-IsingSSM}, and the 
Potts model in two dimensions~\cite{BD1,MOS}, where uniform SSM holds throughout the high-temperature
regime, this uniform SSM property is believed to break down 
for $q\ge 3$ in dimensions $d\ge 3$ at some temperature strictly above the critical one.

For the random-cluster model, one cannot expect uniform SSM to hold 
because boundary conditions can enforce long-range correlations even at very high temperatures (and general boundary conditions can even cause exponentially large mixing times~\cite{BGVfull}). On the other hand,
certain special classes of boundary conditions may exhibit SSM and also be amenable
to proving rapid mixing at $p\ne p_c(q,d)$: this approach was taken by~\cite{BS} in $d=2$ for the class of 
{\it side-homogeneous\/} boundary conditions (i.e., all edges present or all absent
on each side of~$\Lambda_n$), for which SSM followed from planar duality and~\cite{Alexander}. Cutoff was then shown in $d\ge 2$ at very high temperatures $p\ll p_c$~\cite{GS}. However, even SSM for side-homogeneous
boundaries is expected to break down in $d>2$ and $q\ne 2$ at temperatures~$p<p_c$ sufficiently close to the critical point.

In this paper, we prove that SSM for the random-cluster model with respect to {\it any fixed\/} boundary
condition~$\xi$ is enough to imply rapid mixing of the FK dynamics on $\Lambda_n^\xi$.  To the best of our knowledge, such boundary-specific implications were not previously known even for the Ising/Potts dynamics.
(We note that our argument actually adapts easily to prove the same implication for
the Ising Glauber dynamics as well, with potentially interesting implications. See Remark~\ref{rem:ising}.)
\begin{theorem}\label{thm:intro2}
Suppose $d\ge 2, q\ge 1, p$ are such that SSM holds for the random-cluster model on $\Lambda_n^\xi$. Then the mixing time of the FK dynamics on $\Lambda_n^\xi$ is $O(N\log N)$. 
\end{theorem}

Theorem~\ref{thm:intro2} is useful because certain natural boundary conditions
can be shown to satisfy SSM all the way up to the critical point, thus implying
rapid mixing of the FK dynamics in a largest possible range. The two most canonical
boundary conditions to consider are the {\it free\/} condition (denoted~$\zero$)
and the {\it wired\/} condition (denoted~$\one$),
corresponding to all edges in $\mathbb Z^d\setminus \Lambda_n$ being absent and present,
respectively. We are able to prove SSM for these boundary conditions throughout
the high-temperature regime, as well as at sufficiently low temperatures, leading to the following corollary of Theorem~\ref{thm:intro2}.
\begin{corollary}\label{cor:intro2}
For all $d\ge 2$, the mixing time of the FK dynamics is $O(N\log N)$ in
the following settings:
\begin{itemize}
\item[(i)] On $\Lambda_n^\zero$ throughout the high-temperature regime $p<p_c(q,d)$, for all~$q\ge 1$.
\item[(ii)] On $\Lambda_n^\one$ and $\Lambda_n^\zero$ at all sufficiently low temperatures $p>p_0(q,d)$, for all $q\ge 1$.
\end{itemize}
\end{corollary}

\begin{remark}\label{rem:intro-cylinders}
We emphasize that Corollary~\ref{cor:intro2} applies equally to any other class
of boundary conditions for which one can prove the SSM property.  To quote just
one further example, we mention ``cylindrical" boundary conditions, which 
are toroidal in a subset of the coordinate directions and all-free or all-wired in the others. See Remark~\ref{rem:cylinders} for a proof of SSM for these boundaries. 
\end{remark}

\subsubsection{Implications for sampling from the Potts model}\label{subsubsec:sampling} Corollaries~\ref{cor:intro1} and~\ref{cor:intro2} have novel implications
for the existence of sampling algorithms.  Of course, they immediately imply the existence of an efficient
sampling algorithm from the random-cluster distribution~\eqref{eqn:RCgibbs}
in all the specified regimes.  But also, in light
of the intimate connection between the Potts and random-cluster models,
they lead to near-optimal sampling algorithms for the Potts 
distribution~\eqref{eqn:Pgibbs} in several new parameter regimes.  

\begin{corollary}\label{cor:intro3}
For all $d\ge 2$, there is an $O(N\log^3 N)$ time algorithm for sampling from the Gibbs distribution of the $q$-state Potts model in the following scenarios:
\begin{enumerate}[(i)]
    \item On $\bbT_n= (\mathbb Z/n\mathbb Z)^d$ in all the parameter regimes specified in Corollary~\ref{cor:intro1}.
    \item On $\Lambda_n^\zero$ (i.e., no spins are assigned at the boundary) at all high temperatures $\beta<\beta_c(q,d)$, for all~$q\ge 1$.
     \item On $\Lambda_n^{\mathbf k}$ (i.e., all boundary spins have color~$k$) at low enough temperatures $\beta>\beta_0(q,d)$, for all $q\ge 1$.
\end{enumerate}
\end{corollary}

Note that, once
we sample a random-cluster configuration, we can easily convert it to a
Potts configuration using the mechanism described earlier in the introduction.
The extra factor of $O(\log^2 N)$ over the mixing time in Corollary~\ref{cor:intro1} comes from the time needed to implement each
step of the FK dynamics, which requires dynamic maintenance of the connected
components of the configuration~\cite{Wulff-Nilsen}.  

In dimensions $d\ge 3$, Corollary~\ref{cor:intro3} goes well beyond what was previously known.
In particular, it covers the entire high-temperature regime for all~$q$, not only on the
torus but also on~$\Lambda_n^\zero$, the region~$\Lambda_n$ with free
boundary condition. Previously no polynomial time sampling algorithm for the 
Potts model with $q\ge 3$ was known on $\bbT_n$ or $\Lambda_n^\zero$, except at 
sufficiently high
temperatures that the stronger property of SSM uniformly over all boundary
conditions holds~\cite{MaOl2}, or when $q$ is sufficiently large as a function of~$d$~\cite{BCHPT-Potts-all-temp}. Furthermore, in the latter case the polynomial is of high degree, in contrast to the near-optimal runtime of Corollary~\ref{cor:intro3}. We also mention~\cite{BGP} which proves rapid mixing of the Potts Glauber dynamics on general graphs of maximum degree $\Delta$ up to an explicit threshold $\beta \le \frac{1+o_q(1)}{\Delta -1} \log q$, that asymptotically in $q,\Delta$ is roughly half of the critical $\beta_c$ on the lattice of corresponding degree, i.e., $\mathbb Z^{\Delta/2}$.

Corollary~\ref{cor:intro3} also applies at sufficiently low temperatures for all~$q$,
both on the torus and on $\Lambda_n$ with free or 
monochromatic boundary condition.  
(Moreover, subject to the conjecture mentioned earlier that WSM holds at all off-critical temperatures, the restriction to large~$\beta$ is not needed, and our results 
would then cover all parameter regimes where $O(N\log N)$ mixing time of
FK dynamics is to be expected.)
In the low-temperature regime,
polynomial time samplers were available for the Ising model ($q = 2$) \cite{JSIsing,RanWil,GuoJer} on arbitrary graphs,   
but they are of fairly high degree.  
For the Potts model with $q\ge 3$ on $\bbT_n$ and $\Lambda_n^{\mathbf k}$, \cite{BCHPT-Potts-all-temp,HPR-Algorithmic-Pirogov-Sinai} gave the first polynomial time samplers (again of high degree) 
in regimes where a suitable cluster expansion converges; this holds for all $q$ at sufficiently low temperatures, and at all temperatures when~$q$ is sufficiently large as a function of~$d$, but is known to fail at moderate $q$ and $p > p_c(q,d)$. 

\begin{remark}
We conclude this discussion of implications for sampling from the Potts model with a mention of the \emph{Swendsen--Wang (SW) dynamics}, a commonly used global Markov chain that moves on Ising/Potts configurations using the random-cluster representation. 
In~\cite{Ullrich-random-cluster}, comparison results between the SW and random-cluster dynamics showed that their mixing times are within a factor of $O(N)$ of one another. Like the FK dynamics, the analysis of SW is complicated by its non-locality, and for $d \ge 3$ fast mixing on $\bbT_n$ and $\Lambda_n$ were only known on $q=2$~\cite{GuoJer} and for general $q$ at sufficiently high temperatures that \emph{uniform} SSM holds for the Potts model~\cite{BCPSV}. 
By~\cite{Ullrich-random-cluster}, our results imply that under any of the scenarios (i)--(iii) of Corollary~\ref{cor:intro3}, the Swendsen--Wang dynamics has $O(N^2 \log N)$ mixing time. While even polynomial mixing time was not previously known throughout these regimes, the true mixing time should be $\Theta(\log N)$.
\end{remark} 

\subsection{Weak spatial mixing within a phase}\label{subsec:introwsminphase}
As we have mentioned, for the Ising/Potts Glauber dynamics, the notions of WSM and SSM have been the dominant tools for obtaining mixing time bounds at high temperatures $\beta<\beta_c$; as soon as $\beta \ge \beta_c$ both these properties break down, and indeed their mixing times on $\bbT_n  = (\mathbb Z/n\mathbb Z)^d$ become exponentially slow at all $\beta>\beta_c(d)$~\cite{Pisztora96,Bodineau05}. The slowdown is due
to the emergence of an exponentially tight bottleneck between the phases (each corresponding
to dominance by one of the $q$ colors), which coexist at low temperatures.
A generally believed paradigm (which turns out to be quite hard to prove) 
is that the mixing time, when initialized from an appropriate {\it random mixture\/}
of configurations that are representative of each of these phases\footnote{
We note that the notion of mixing time from a specified (possibly random) initialization is slightly weaker than the standard worst-case mixing time, in that the usual ``boosting" property does not apply. See Section~\ref{subsec:mixingtimes} for details.}---in the Ising case, the all-$+1$ and all-$-1$ configurations\footnote{By convention, in the case of the Ising model $q=2$, the two states are usually taken to be $\{-1,+1\}$ rather than $\{1,2\}$} with probability $1/2$ each---should in fact be fast.  A version of this paradigm was recently established in~\cite{GhSi21} for the  Ising model on $\bbT_n$ in its entire low-temperature coexistence regime using a new notion of spatial mixing known as \emph{WSM within a phase}. 
Informally, WSM within a phase requires exponential decay of correlations
between sets~$A,B$ on $\bbT_n$ as before, but now under the Gibbs measure restricted to a phase (e.g., to configurations with majority~$+1$ spins), and only when the configuration on $A$ is set to the extremal configuration of that phase (i.e., 
the all-$+1$ configuration).

Recall that, for the random-cluster model, WSM is expected to hold at all off-critical temperatures $p\ne p_c(q,d)$; the only phase coexistence regime is then {\it exactly at\/} the critical point $p_c(q,d)$, and actually only when $q$ is large enough.\footnote{In $d=2$, there is coexistence at $p_c(q,d)$ as soon as $q>4$~\cite{DGHMT}, and in large dimension as soon as $q>2$. For smaller values of~$q$, the random-cluster model has \emph{no} coexistence regime, and the mixing time should be polynomial even at $p_c$~\cite{LS-critical-2D-Ising,GL2}.} The coexistence here is between the high- and low-temperature phases, corresponding to non-existence and existence of a giant component (called the ``free" or ``disordered" and ``wired" or ``ordered" phases, respectively). Indeed, at $p_c(q,d)$ when $q$ is large, there is an exponential bottleneck between the free and wired phases on $\bbT_n$, and the FK dynamics on $\bbT_n$ slows down exponentially~\cite{BCT,GL1}. (Technically in $d\ge 3$,~\cite{BCT} only applies for integer $q$: for completeness, in Appendix~\ref{app:slow-mixing} we deduce slow mixing for all $d\ge 2$ and general $q\ge q_0$ at $p = p_c + o(1/n)$.)
By analogy with the Ising model, one might conjecture that this slowdown can be overcome by starting the FK dynamics in an appropriate mixture of extremal representatives from its two phases, namely the all-free and all-wired configurations. 

In this paper we define a notion of WSM within a phase for the random-cluster model (see Definition~\ref{def:WSM-within-a-phase} for the details; this definition is technically
more complicated than in the Ising case due to the lack of symmetry between the free and wired phases). We then use this notion to verify the above intuition by first establishing the following upper bound on the mixing time of the FK dynamics \emph{within} each of its phases when initialized from the corresponding extremal configuration. We capture this by considering the dynamics ``restricted to a phase", by which we mean that transitions to the other phase are rejected.

\begin{theorem}\label{thm:intro3}
Suppose $d \ge 2, q\ge 1$ and $p$ are such that there is an exponential bottleneck between the free and wired phases, and WSM within the wired phase
holds.  Then the mixing time of
the FK dynamics on $\bbT_n$ restricted to the wired phase,
starting from the all-wired configuration, is $O(N \log N \cdot \tmix(\Lambda_{C\log N}^{\one}))$.  
A symmetrical statement holds for the free phase.
\end{theorem}
Note that this theorem reduces the mixing time of the FK dynamics restricted
to the wired (resp., free) phase initialized from the extremal 
all-wired (resp., all-free) configuration to that of the (worst-case initialization) mixing time of the dynamics
with wired (resp., free) boundary condition, but at an exponentially smaller scale.
The mixing time with wired boundary conditions is expected to be 
polynomial in the size of the region (by analogy with the
conjectured behavior of the low-temperature Ising model with $+1$ or $-1$ boundary conditions), in which case the mixing time bound in Theorem~\ref{thm:intro3} would be $O(N(\log N)^{C'})$. However, even without this conjecture we can deduce, via a
crude bound on the mixing time at logarithmic scale, that
the mixing time implied by Theorem~\ref{thm:intro3} is at most quasi-polynomial in all
dimensions (and indeed quasi-linear in $d=2$ using a more refined such bound  from~\cite{GL1}). 

Finally, we are able to apply Theorem~\ref{thm:intro3} to establish a version of the
above conjecture about mixing from a random phase initialization in the FK dynamics.  

\begin{theorem}\label{thm:intro4}
Fix $d\ge 2$, $q\ge q_0(d)$ and $p=p_c(q,d)$. The FK dynamics on $\bbT_n$, initialized from a $(\frac{1}{q+1},\frac{q}{q+1})$-mixture of the all-free and the all-wired configurations, satisfies the following: 
\begin{itemize}
    \item If $d=2$, the mixing time is $N^{1+o(1)}$;
    \item If $d\ge 3$, the mixing time is $N^{O((\log N)^{d-2})}$.
\end{itemize}
\end{theorem}

Given Theorem~\ref{thm:intro3}, the core of the proof of Theorem~\ref{thm:intro4} is to show that, for $q\ge q_0(d)$,
the random-cluster model satisfies WSM within both the wired and free phases at~$p_c$
(see Lemma~\ref{lem:large-q-wsm-within-phase}). Theorem~\ref{thm:intro3},
together with crude mixing time estimates for extremal boundary
conditions at logarithmic scale, then
implies the stated mixing time bounds within each phase at~$p_c$.
{\it Simultaneous\/} mixing within each phase at~$p_c$
then easily implies mixing for the entire Gibbs distribution~\eqref{eqn:RCgibbs}
at~$p_c$, provided the all-free and all-wired configurations are correctly weighted in the initialization.
We recall that, in contrast, the mixing time on $\bbT_n$ from a worst-case initialization 
suffers an exponential slowdown at~$p_c$.

A potential shortcoming of Theorem~\ref{thm:intro4} is that a priori knowledge of $p_c$ is a delicate matter (there is no exact expression for it in $d\ge 3$). Relatedly, if $p$ is not exactly $p_c(q,d)$ but is microscopically close, the weights accorded to the wired and free phases may both be $\Omega(1)$ but different from the values $(1,q)$ exactly at~$p_c$. 
It turns out that WSM within a phase holds for the wired phase {\it uniformly\/} over $p\ge p_c - o(1/n)$ and for the free phase {\it uniformly\/} over $p\le p_c - o(1/n)$. These behaviors can therefore be stitched together to ensure that sampling from random phase initializations of the form in Theorem~\ref{thm:intro4} satisfies the above bounds at \emph{all} temperatures, provided the weighting of the free and wired initializations is accordingly adjusted: see Corollary~\ref{cor:random-phase-initialization-general-weight}. 
In Lemma~\ref{lem:learning-weights} we describe an MCMC algorithm whose run-time is at most $N^{O((\log N)^{d-2})}$ to {\it learn\/} the relative weights of the two phases---when both are non-negligible---to sufficient accuracy by random sampling within each phase, along sequences of $p$ that approach $p_c$ from above and below.

\begin{remark}
As stated earlier, the quasi-polynomial mixing times in Theorem~\ref{thm:intro4}
immediately become polynomial (or indeed $\tilde O(N)$) if as conjectured the
mixing time of FK dynamics with wired and free boundary is shown to be polynomial at $p_c$. Similarly, the time to approximately learn the relative phase weights for $p\approx p_c$ would become some low-degree polynomial in~$N$.
\end{remark}

In the same fashion as Corollary~\ref{cor:intro3}, our Theorem~\ref{thm:intro4} 
has implications for sampling from the Potts model at, and around, its critical
point $\beta_c = -\ln(1-p_c)$.
In~\cite{BCHPT-Potts-all-temp}, a polynomial-time algorithm for sampling from the Potts model exactly at $\beta_c(q,d)$ was provided for $q\ge q_0(d)$ using a large-$q$ cluster expansion. 
(As the authors note, that algorithm suffers from the shortcomings
mentioned above concerning the exact identification of~$\beta_c$.)
As evidenced by the $d=2$ case, our approach has the potential for yielding essentially optimal-time sampling algorithms at $p_c(q,d)$. Furthermore, we anticipate that the approach should work at all $q$ for which there is coexistence at $p_c(q,d)$, as this
should be sufficient to ensure WSM within the phases.

\section{Preliminaries and notation}\label{sec:preliminaries}
In this section, we introduce our main notation and recall important properties of the random-cluster model and associated Markov chains. We refer the reader to~\cite{Grimmett} and~\cite{LP} for more on these topics.  

\subsection{The random-cluster model on \texorpdfstring{$\mathbb{Z}^d$}{Zd}}\label{subsec:rcmodel}
The random-cluster model on a finite graph $G=(V,E)$
with parameters $p\in[0,1]$ and $q\ge 1$ is defined in~\eqref{eqn:RCgibbs}.
In a configuration $\omega\subseteq E$, if edge~$e$ belongs to~$\omega$
we write $\omega(e)=1$ and call~$e$ {\it wired\/} or {\it open}; else we write 
$\omega(e)=0$ and call~$e$ {\it free\/} or {\it closed}.  
If $x,y$ are in the same connected component of the sub-graph $(V,\omega)$, we write $x\xleftrightarrow[]{\omega} y$. 

Throughout the paper, we will focus on the case where~$G$ is a rectangular subgraph
of the $d$-dimensional integer lattice $\mathbb Z^d$ with nearest-neighbor edges
$E(\mathbb{Z}^d) = \{\{u,v\}: u,v\in \mathbb{Z}^d, d_1(u,v) = 1\}$, where $d_1(\cdot,\cdot)$ is $\ell_1$ distance. The subsets of $\mathbb Z^d$ we will be interested in are
\begin{align*}
    \Lambda_n & = [- \tfrac n2 ,  \tfrac n2]^d \cap \mathbb Z^d\,,
\end{align*}
with induced edge set denoted $E(\Lambda_n)$.  The (inner) boundary vertices
of~$\Lambda_n$, denoted $\partial \Lambda_n$, are those vertices in $\Lambda_n$
that have neighbors in $\mathbb Z^d \setminus \Lambda_n$.

\subsubsection{Boundary conditions}\label{subsubsec:bc}

A random-cluster \emph{boundary condition} $\xi$ on $\Lambda_n$ is a partition of~$\partial\Lambda_n$ into an arbitrary number of subsets, such that the vertices in each subset are identified with one another. The random-cluster distribution with boundary condition~$\xi$, denoted $\pi_{\Lambda_n^\xi,p,q}$, is the same as in~\eqref{eqn:RCgibbs}, except that the set $\mathsf{Comp}(\omega)$ is replaced by $\mathsf{Comp}(\omega;\xi)$, the set of connected
components with this vertex identification.
Thus the boundary condition can be viewed as ``ghost wirings" of  vertices in the same subset of~$\xi$.

The \emph{free} boundary condition, denoted $\xi = \zero$, is the one whose partition of $\partial \Lambda_n$ consists only of singletons.
The \emph{wired} boundary condition, $\xi = \one$, puts all vertices of~$\partial\Lambda_n$ in the same subset. There is a natural partial order on boundary conditions, given by $\xi \le \xi'$ iff $\xi$ is a refinement of~$\xi'$. The wired/free boundary conditions are then the maximal/minimal ones under this order.

The \emph{periodic} boundary condition on~$\Lambda_n$ corresponds to the partition
that pairs up opposite points on $\partial \Lambda_m$. Under this identification of vertices, the random-cluster model on~$\Lambda_n$ is the same as that on the torus $\bbT_n=(\mathbb Z/n\mathbb Z)^d$. 
Finally, we will make use of boundary conditions induced by an arbitrary 
configuration~$\eta$ on $\mathbb Z^d \setminus \Lambda_n$, corresponding
to the partition in which $x,y\in \partial\Lambda_n$ are in the same subset 
if and only if $x\xleftrightarrow[]{\eta} y$. 

An important property of the random-cluster model that
we will use extensively 
is \emph{monotonicity in boundary conditions}~\cite[Lemma 4.14]{Grimmett}: 
for any two boundary conditions $\xi\le \xi'$, 
 we have $\pi_{\Lambda_n^{\xi'},p,q} \succcurlyeq \pi_{\Lambda_n^{\xi},p,q}$ 
 where $\succcurlyeq$ denotes stochastic domination.  

The above definition of boundary conditions and monotonicity in boundary conditions carry over naturally to general domains $A \subset \mathbb Z^d$ with associated boundary $\partial A$.

\subsection{Phase transition on \texorpdfstring{$\mathbb Z^d$}{Zd}} On $\mathbb Z^d$, the random-cluster model is well-known to undergo a percolation phase transition at some $p_c(q,d)$. Classical theory (see e.g.,~\cite[Section 5]{Grimmett}) implies that for all $p<p_c$, the probability of a giant component (a component having $\Theta(N)$ many vertices) under $\pi_{\Lambda_n}$ is $o(1)$, whereas when $p>p_c$ the probability of a giant component is $1-o(1)$. It turns out that the phase transition is sharp and for all $p<p_c$, the component sizes have exponential tails~\cite{ABF87,BD1,DCRT19}. When $p>p_c$, it is expected that all non-giant components have exponential tails; this is known in $d=2$ by planar duality, in the $q=2$ case by~\cite{Pisztora96,Bodineau05}, and when $q\ge q_0$ for some $q_0(d)$ by a technique known as Pirogov--Sinai theory~\cite{BorgsKotecky}. 

At the critical point $p_c$, the behavior is very rich and itself exhibits a transition as one varies $q$. Namely, there exists a particular $q_c(d)$ such that when $q\le q_c$, the behavior at $p_c$ should be like that of the Ising model, in that component sizes have polynomial decay on their sizes, whereas when $q>q_c$, there is phase coexistence. In this context, that means the model at $p_c$ is in a mixture of two possibilities, one mirroring the $p<p_c$ behavior, and one mirroring the $p>p_c$ behavior. This coexistence behavior is known rigorously in $d=2$, where $q_c=4$~\cite{DST,DGHMT}, and in general dimension when $q\ge q_0$~\cite{BorgsKotecky}.

\subsection{Mixing times of Markov chains}\label{subsec:mixingtimes}
Consider a (discrete-time) Markov chain with transition matrix $P$ on a finite state space $\Omega$, reversible with respect to an invariant distribution $\pi$; denote the chain initialized from $x_0\in \Omega$ by $(X_t^{x_0})_{t \in \mathbb N}$. Its (worst-case) \emph{mixing time} is given by 
\begin{align}\label{eq:def-mixing-time}
    \tmix(\epsilon) := \min \big\{t: \max_{x_0 \in \Omega} \| \mathbb P(X_t^{x_0} \in \cdot ) - \pi\|_\tv \le \epsilon\big\}\,,
\end{align}
where $\|\mu - \nu\|_\tv$ is total variation distance.  
Typically the mixing time is defined specifically as $\tmix(1/4)$, 
since for any $\delta$ one has $\tmix(\delta) \le \tmix(1/4) \log (2\delta^{-1})$.
To bound the mixing time, it suffices to bound the \emph{coupling time}; i.e., if we construct a coupling $\mathbb P$ of the steps of the chain such that for each $x_0,y_0 \in \Omega$, we have 
$\mathbb P(X_T^{x_0} \neq X_T^{y_0}) \le 1/4$, then $\tmix \le T$.  

For a probability distribution~$\nu$ over~$\Omega$, we shall
also talk about the \emph{mixing time with initialization}~$\nu$, 
which is defined as in~\eqref{eq:def-mixing-time}
but without the maximization over~$x_0$ and with~$x_0$ chosen according to~$\nu$.
 In this case the above ``boosting" does not apply; however, in this paper all
of our mixing time bounds from specific initializations suffice to achieve
total variation distance $\epsilon$ equal to any desired inverse polynomial in~$N$.

\subsection{The FK dynamics}
Recall the definition of the random-cluster (FK) dynamics from the introduction. In the presence of boundary conditions $\xi$, the only change is that in the update
rule~\eqref{eq:update-rule} the condition for $e_{t+1}$ to be a bridge is 
determined by whether its presence changes $\mathsf{Comp}(X_t^{\omega_0}; \xi)$. 

\subsubsection{Continuous time}
It will be convenient for us to prove our mixing time bounds for the
FK dynamics on~$\Lambda_n$ in {\it continuous time} rather than discrete time. 
This is defined by assigning every edge in $E(\Lambda_n)$ an i.i.d.\ rate-1 Poisson clock; if the clock at~$e$ rings at time $t_i$, we update $X_{t_i}^{\omega_0}(e)$ according to the update rule~\eqref{eq:update-rule}. 
This definition is the same as taking the continuous-time Markov chain with heat kernel $H_t=e^{t(I-P)}$, with $P$ being the transition matrix for the discrete-time FK dynamics. 
Recall (e.g.,~[Theorem 20.3]\cite{LP}) that 
\begin{align}\label{eq:continuous-time-discrete-time-comparison}
    C^{-1} |E(\Lambda_n)| \tmix^{\textrm{cont}} \le \tmix  \le  C |E(\Lambda_n)|\tmix^{\textrm{cont}}
\end{align}
for some absolute constant $C$, where $\tmix^{\textrm{cont}}$ is the mixing time of the continuous-time dynamics. (Note that the constant $C$ here depends on the minimum probability in~\eqref{eq:update-rule} of keeping the edge status unchanged, which is a function only of $p,q$.) Thus, it suffices to prove our mixing time bounds in continuous
time, reduced by a factor of $E(\Lambda_n)$. (The comparison also holds for the mixing time from a specified initialization.) Abusing notation slightly, from this point onwards we let $(X_t^{\omega_0})_{t\ge 0}$ and all other random-cluster dynamics chains we consider be the continuous-time versions as defined above.

\subsubsection{Monotonicity and the grand coupling}\label{subsec:prelim-grand-coupling}
The FK dynamics can be seen to be monotone in the following sense: if $\omega_0 \ge \omega_0'$, then the law of $X_t^{\omega_0}$ stochastically dominates the law of $X_{t}^{\omega_0'}$.
Moreover, there is a standard choice of \emph{grand monotone coupling} (using
the same instantiation of Poisson clock rings, and the same source of 
randomness for the edge updates)
which simultaneously couples all FK dynamics chains $(X_{t,A^\xi}^{\omega_0})$ indexed by their initial configuration $\omega_0$, domain $A\subset \mathbb Z^d$, and boundary condition $\xi$. 
It is not difficult to verify that the coupling is monotone in the sense that if $\omega_0 \ge \omega_0'$ and $\xi \ge \xi'$, then $X_{t,A^\xi}^{\omega_0}\ge X_{t,A^{\xi'}}^{\omega_0'}$ for all $t \ge 0$: see e.g.,~\cite[Section 8.3]{Grimmett}. 

\subsubsection{The restricted FK dynamics}\label{subsec:restricted-fk-dynamics}
A crucial tool in our analysis will be the FK dynamics 
\emph{restricted to} an increasing event $\widehat \Omega$ (symmetrically, to a decreasing event $\widecheck \Omega$).
This chain, denoted $(\widehat X_t^{\omega_0})_{t\ge 0}$ (symmetrically, $(\widecheck X_t^{\omega_0})_{t\ge 0}$) is defined exactly
as before, except that 
if the update in~\eqref{eq:update-rule} would cause~$\omega$ not to be 
in~$\widehat \Omega$ (symmetrically, in~$\widecheck \Omega$) we leave $\omega(e)$ unchanged.
It is easy to check that the Markov chains $(\widehat X_t)$ and $(\widecheck X_t)$ are
reversible w.r.t.\ $\widehat \pi  = \pi(\cdot \mid \widehat \Omega)$ and $\widecheck \pi = \pi(\cdot \mid \widecheck \Omega)$, respectively. 

\subsection{Notational disclaimers}
The letter $C>0$ will be used frequently, indicating the existence of a constant that is independent of $r,n,m$ etc., but that may depend on
$p,d,q$ and may differ from line to line. We use~$O$, $o$ and~$\Omega$ notation in the same manner, where the hidden constants may depend on $p,d,q$.

\section{Spatial and temporal mixing for off-critical random-cluster models}
In this section we establish the implications of WSM and of SSM with a specific boundary condition for the mixing time of the FK dynamics on, respectively, the torus and a box with that specific boundary condition. Namely, the purpose of this section is to establish Theorems~\ref{thm:intro1} and~\ref{thm:intro2}. 

We begin by giving the formal definitions of WSM, and of SSM with boundary condition $\xi$ for the random-cluster model.  We formulate these for regions that are boxes
in~$\mathbb{Z}^d$.

\begin{definition}\label{def:WSM}
    We say the random-cluster model satisfies {\em weak spatial mixing (WSM)\/} if, for every $r$,
    \begin{align}\label{eq:WSM}
        \|\pi_{\Lambda_r^\one} (\omega(\Lambda_{r/2}) \in \cdot) - \pi_{\Lambda_r^\zero}(\omega(\Lambda_{r/2})\in \cdot) \|_\tv \le Ce^{ - r/C}\,,
    \end{align}
for some constant $C>0$ (which may depend on $d$, $q$ and~$p$).
\end{definition}

To phrase the property of SSM with a specific boundary condition precisely, given a box $\Lambda_n$ with boundary conditions $\xi$, define $B_{m,e}$ to be the graph with vertex set $\{v\in \Lambda_n: d_\infty(v,e)\le m\}$ and induced edge set $E(B_{m,e})$. When understood from context, $B_{m,e}$ will stand in for its edge set for readability. If $\eta$ is a
configuration on $E(\Lambda_n)\setminus E(B_{m,e})$, then by $B_{m,e}^\eta$ we mean the graph $B_{m,e}$ with boundary conditions induced by~$\eta$ together with~$\xi$. 

\begin{definition}\label{def:ssm-per-bc}
    We say the random cluster model on $\Lambda_n^\xi$ satisfies {\em SSM with constant $C$\/} if, for all $e\in E(\Lambda_n)$ and all $m\le n/2$, we have 
    \begin{align}\label{eq:SSM}
        \|\pi_{B_{m,e}^\one}(\omega(B_{m/2,e})\in \cdot) - \pi_{B_{m,e}^\zero}(\omega(B_{m/2,e})\in \cdot)\|_\tv \le C e^{ - m/C}\,.
    \end{align}
\end{definition}

\begin{remark}
   If Definitions~\ref{def:WSM}--\ref{def:ssm-per-bc} hold for some constant $C$, then there exists a constant $C'(C,p,q)>0$ such that they hold with the right-hand sides of~\eqref{eq:WSM} and~\eqref{eq:SSM} replaced by $e^{ - r/C'}$ and $e^{ - m/C'}$, respectively. This can be readily checked using the fact that for any fixed $r,m$, the TV-distances are bounded strictly away from $1$ by forcing all edges to take the same state. In the sequel, we will sometimes use this modified form for convenience. 
\end{remark}

For convenience, we now restate Theorems~\ref{thm:intro1} and~\ref{thm:intro2} in the
language of continuous time. (The restatement of Theorem~\ref{thm:intro2} is also more precise about quantification and dependencies on the SSM constant, because a specific boundary condition may only make sense for a single box-size $n$.)
\par\medskip\noindent
{\bf Theorem~\ref{thm:intro1} [formal].}
{\it Fix $(p,q,d)$. If the random-cluster model satisfies WSM, then the mixing time of the continuous-time FK dynamics on $\bbT_n = (\mathbb Z/n\mathbb Z)^d$ is $O(\log N)$.}
\par\medskip

\par\smallskip\noindent
{\bf Theorem~\ref{thm:intro2} [formal].}    
{\it Fix any $C_\star>0$ and $(p,q,d)$. There exists $n_0$ such that for all $n\ge n_0$, if the random cluster model on $\Lambda_n^\xi$ has SSM with constant $C_\star$, then the mixing time of the continuous-time FK dynamics on $\Lambda_n^\xi$ is~$O(\log N)$.}
\par\medskip

As mentioned in the introduction, it is fairly straightforward to deduce
Theorem~\ref{thm:intro1} using the 
classical arguments of Martinelli and Olivieri~\cite{MaOl1},
which, as noted by Harel and Spinka~\cite{HarelSpinka}, extend naturally to the random-cluster model.  Our more novel
contribution is therefore Theorem~\ref{thm:intro2}, which
can be seen as an extension of the implication of Theorem~\ref{thm:intro1} to the more delicate 
setting of a fixed boundary condition.  Accordingly, we shall devote the majority of this
section to proving Theorem~\ref{thm:intro2}, and will
then for completeness briefly indicate 
in Section~\ref{subsec:wsm-torus} 
how to extract Theorem~\ref{thm:intro1} from our proof.

Our approach to proving Theorem~\ref{thm:intro2} is a finite-volume analogue of the space-time recursion of~\cite{MaOl1}. Whereas that argument 
shows that WSM implies exponential relaxation in the infinite volume~$\mathbb{Z}^d$, 
to the best of our knowledge it has not until now been applied in the presence of
specific boundary conditions. Previous proofs for fast mixing in the presence of boundary conditions have gone through recursive schemes that require notions of SSM that apply uniformly over \emph{all}, or at least all side-homogeneous, boundary conditions. Recall that there are parameter regimes in $d\ge 3$ where such uniform SSM is not expected to hold, but the random-cluster model on $\Lambda_n^\zero$ or $\Lambda_n^\one$ does satisfy SSM in the form of Definition~\ref{def:ssm-per-bc}.

\begin{remark}\label{rem:ising}
Since the proof of Theorem~\ref{thm:intro2} primarily used monotonicity, it would equally work for the Glauber dynamics for the Ising model on $\Lambda_n$ in the presence of boundary conditions $\eta\in \{\pm 1\}^{\mathbb Z^d\setminus\Lambda_n}$, assuming only SSM for $\Lambda_n^\eta$ and not uniformly over all boundary conditions. It was recently shown in~\cite{DSS-IsingSSM} that the Ising model has SSM uniformly over all boundary conditions at all $\beta<\beta_c(d)$ in the absence of an external field. 
However, there are parameter regimes (e.g., low-temperatures with positive external field in $d\ge 3$) where the Ising model is known to have WSM~\cite{Ott-weak-mixing} but \emph{not} SSM uniformly over all boundary conditions (an example bad boundary condition is given in~\cite{MOS}). 
All the same, even in this parameter regime there are classes of boundary conditions---most obviously the all-$+1$ condition---that do satisfy SSM.
Our results can therefore be used to deduce fast mixing for the Ising Glauber dynamics in such settings, may not have followed from known technology.
\end{remark}

\subsection{A recurrence for the disagreement probability at a single edge}

Let us begin with some notation. Throughout this section, $n$ and $\xi$ will be fixed such that $\Lambda_n^\xi$ satisfies SSM with constant $C_0$, and $n\ge n_0(C_0)$ to be determined later. 
For $m\le n/2$, for $r\le m$, define the event 
\begin{align*}
    \cE_t(m,r,e) := \big\{X_{t,B_{m,e}^\one}^\one(B_{r,e}) \ne X_{t,B_{m,e}^\zero}^\zero(B_{r,e})\big\}\,,
\end{align*}
and under the grand coupling of $X_{t,B_{m,e}^\one}^\one,X_{t,B_{m,e}^\zero}^\zero$, let 
\begin{align*}
    \phi_{m,t}(e):= \mathbb P(\cE_t({m,0,e})) = \mathbb P\big(X_{t,B_{m,e}^\one}^\one(e) \ne X_{t,B_{m,e}^\zero}^\zero(e)\big)\,.
\end{align*}
By monotonicity of the grand coupling, we next observe that 
\begin{align}\label{eq:phi-monotonicity-ssm}
    \phi_{m,s}(e)\ge \phi_{l,t}(e) \qquad \mbox{for all $m\le l$ and $s\le t$}
\end{align}
Indeed, $B_{l,e}^\one$ stochastically dominates $B_{m,e}^\one$ if $m\le l$ regardless of $\xi$, and vice versa with $\zero$ boundary. For the monotonicity in time, notice that by the censoring lemma of~\cite[Theorem 1.1]{PWcensoring}, $X_{t,B_{m,e}^\one}^\one$ is stochastically decreasing in time, and $X_{t,B_{m,e}^\zero}^{\zero}$ is stochastically increasing in time, so 
\begin{align*}
    \phi_{m,s}(e) & \! = \! \mathbb P (X_{s,B_{m,e}^\one}^\one(e) =1) - \mathbb P (X_{s,B_{m,e}^\zero}^\zero(e) =1) \! \ge \mathbb P (X_{t,B_{m,e}^\one}^\one(e) =1) - \mathbb P (X_{t,B_{m,e}^\zero}^\zero(e) =1) \!= \!\phi_{m,t}(e)\,.
\end{align*}
Define now 
\begin{equation}\label{eq:phi-m-t-ssm}
    \phi_{m,t} :=\max_{e\in E(\Lambda_n)}\phi_{m,t}(e)\,.
\end{equation}
Observe that~\eqref{eq:phi-monotonicity-ssm} implies the same monotonicity relation for $\phi_{m,t}$. 
The main result of this subsection, and indeed the main input into establishing Theorem~\ref{thm:intro2} is the following. 

\begin{proposition}\label{prop:phi-recurrence-ssm}
Suppose that the random-cluster model on $\Lambda_n^\xi$ satisfies SSM with constant $C_\star$. Then for every $t\ge 0$, every $m\le n/4$, and every $r\le m$,
\begin{align}\label{eqn:recurrence-ssm}
    \phi_{2m,2t} \le  e^{ - r/C_\star}+  d(2r)^d \phi_{m,t}^2\,.
\end{align}
\end{proposition}

\begin{proof}
Fix $t$ and any edge $e\in E(\Lambda_n)$, and consider the quantity $\phi_{2m,2t}(e)$. We can of course express this as
\begin{align}\label{eq:P-2t-f-splitting-ssm}
    \phi_{2m,2t} = \mathbb P(\cE_{2t}(2m,0,e), \cE_t(2m,r,e)^c) + \mathbb P(\cE_{2t}(2m,0,e), \cE_t(2m,r,e)).
\end{align}
Let us consider these two terms separately. For the first term we notice that, under the grand coupling, on the event $\cE_t(2m,r,e)^c$ the chains with all possible initializations agree on $B_{r,e}$, so we have $$X_{t,B_{2m,e}^\one}^\one (B_{r,e})= X_{t,B_{2m,e}^\one}^{\pi_{B_{2m,e}^\one}}(B_{r,e})\,.$$
By monotonicity in both initialization and boundary conditions, this implies 
$$X_{t,B_{2m,e}^\one}^\one (B_{r,e})\le X_{t,B_{r,e}^\one}^{\pi_{B_{r,e}^\one}}(B_{r,e})\,.$$
The analogous (reverse) inequality holds for $X_{t,B_{2m,e}^\zero}^\zero(B_{r,e})$. In particular, on $\cE_t(2m,r,e)^c$, we have 
$$
X_{t,B_{r,e}^\one}^{\pi_{B_{r,e}^\one}}(B_{r,e}) \ge X_{t,B_{2m,e}^\one}^\one(B_{r,e})\ge X_{t,B_{2m,e}^\zero}^\zero(B_{r,e}) \ge X_{t,B_{r,e}^\zero}^{\pi_{B_{r,e}^\zero}}(B_{r,e})\,.
$$
By monotonicity of the FK dynamics and the Markov property, on $\cE_t(2m,r,e)^c$, this ordering is maintained for all times beyond $t$, so that in particular
$$
X_{2t,B_{r,e}^\one}^{\pi_{B_{r,e}^\one}}(e) \ge X_{2t,B_{2m,e}^\one}^\one(e)\ge X_{2t,B_{2m,e}^\zero}^\zero(e) \ge X_{2t,B_{r,e}^\zero}^{\pi_{B_{r,e}^\zero}}(e)\,.
$$
Therefore, 
\begin{align*}
    \mathbb P(\cE_{2t}(2m,0,e), \cE_t(2m,r,e)^c) & = \mathbb E[(X_{2t,B_{2m,e}^\one}^\one(e)- X_{2t,B_{2m,e}^\zero}^\zero(e)) \mathbf 1\{\cE_t(2m,r,e)^c\}]\nonumber \\ 
    & \le \mathbb E[(X_{2t,B_{r,e}^\one}^{\pi_{B_{r,e}^\one}}(e)- X_{2t,B_{r,e}^\zero}^{\pi_{B_{r,e}^\zero}}(e)) \mathbf 1\{\cE_t(2m,r,e)^c\}]\,. \nonumber
\end{align*}
Under the grand coupling, $X_{2t,B_{r,e}^\one}^{\pi_{B_{r,e}^\one}}(e)- X_{2t,B_{r,e}^\zero}^{\pi_{B_{r,e}^\zero}}(e)\ge 0$, so we can now drop the indicator to get 
\begin{align}
\mathbb P(\cE_{2t}(2m,0,e), \cE_t(2m,r,e)^c) & \le \mathbb E[X_{2t,B_{r,e}^\one}^{\pi_{B_{r,e}^\one}}(e)] - \mathbb E[ X_{2t,B_{r,e}^\zero}^{\pi_{B_{r,e}^\zero}}(e)] \nonumber\\
& = \pi_{B_{r,e}^\one}(\omega_{e} = 1) - \pi_{B_{r,e}^\zero}(\omega_{e} = 1)\le \exp(-r/C_\star) \label{eq:2.2temp1-ssm}\,,
\end{align}
where the last inequality follows from our assumption that the random-cluster model on $\Lambda_n^\xi$ satisfies SSM with constant $C_\star$ as $r\le n/2$ (recall Definition~\ref{def:ssm-per-bc}). 

We now turn to the second term in~\eqref{eq:P-2t-f-splitting-ssm}. Clearly we can write
\begin{align}
    \mathbb P(\cE_{2t}(2m,0,e), \cE_t(2m,r,e)) = \mathbb P(\cE_{2t}(2m,0,e) \mid \cE_t(2m,r,e)) \mathbb P(\cE_t(2m,r,e)). \label{eq:2.2temp2-ssm}
\end{align}
A union bound gives
\begin{align*}
    \mathbb P(\cE_t(2m,r,e)) \le |E(B_{r,e})| \max_{f\in E(B_{r,e})} \mathbb P\Big(X_{t,B_{2m,e}^\one}^\one(f) \ne X_{t,B_{2m,e}^\zero}^\zero(f)\Big)\,.
\end{align*}
Since $r\le m$, for every $f\in B_{r,e}$, $B_{m,f}\subset B_{2m,e}$, so that by monotonicity
\begin{align*}
    \mathbb P(\cE_t(2m,r,e)) \le |E(B_{r,e})| \max_{f\in E(B_{r,e})} \mathbb P\Big(X_{t,B_{m,f}^\one}^\one(f) \ne X_{t,B_{m,f}^\zero}^\zero(f)\Big) & = |E(B_{r,e})|\max_{f\in E(B_{r,e})} \phi_{m,t}(f) \\
    & \le d(2r)^d \phi_{m,t} \,.
\end{align*}
At the same time, by the Markov property and monotonicity of the grand coupling, taking a worst case over the realizations of the coupling in the time interval $[0,t)$, we get 
\begin{align*}
    \mathbb P (\cE_{2t}(2m,0,e) \mid \cE_t(2m,r,e)) \le \max_{A_t\in \cF_t} \mathbb P (\cE_{2t}(2m,0,e)  \mid A_t) \le \mathbb P(\cE_{t}(2m,0,e)) = \phi_{2m,t}(e)\,,
\end{align*}
where $\cF_t$ is the filtration defined by the grand coupling up to time $t$. 
The right-hand side is in turn at most $\phi_{m,t}(e)\le \phi_{m,t}$ by~\eqref{eq:phi-monotonicity-ssm}. 
Plugging these bounds into~\eqref{eq:2.2temp2-ssm} we get 
\begin{align*}
    \mathbb P (\cE_{2t}(2m,0,e), \cE_{t}(2m,r,e)) \le d(2r)^d \phi_{m,t}^2\,.
\end{align*}
Together with the bound~\eqref{eq:2.2temp1-ssm} on the first term in~\eqref{eq:P-2t-f-splitting-ssm}, we obtain the desired bound~\eqref{eqn:recurrence-ssm}. 
\end{proof}

\subsection{Exponential decay of disagreement probability at a single edge}
We next claim that Proposition~\ref{prop:phi-recurrence-ssm} implies exponential decay of $\phi_{m,t}$ in space and time. 

\begin{corollary}\label{cor:phi-exp-decay-ssm}
For every $C_\star$, there exists $C(C_\star,p,q,d)>0$ such that if the random-cluster model on $\Lambda_n^\xi$ satisfies SSM with constant $C_\star$, then we have
\begin{align*}
     \phi_{t,t} \le C e^{ - t/C} \qquad \mbox{for all $t\le n/2$}\,.
\end{align*}
\end{corollary}

\begin{proof}
The corollary will follow by setting $a_k = \phi_{k,k}$ and applying the following lemma. 

\begin{lemma}\label{lem:recurrence-solution}
Fix $d$ and $C_\star$. Suppose  $0\le a_k\le 1$ is a non-increasing (in $k$) sequence. There exists $\epsilon_0(C_\star,d)$ such that if $a_{k_0}\le \epsilon_0$ for some $k_0$, and $a_k$ satisfies 
\begin{align*}
    a_{2k}\le d(2r)^d a_k^2 + e^{ - r/C_\star}\,, \qquad \mbox{for all $r\le k\le n/2$}\,,
\end{align*}
then there exists $C = C(k_0,C_\star,d)$ such that $a_k\le Ce^{ - k/C}$ for all $k\le n/2$. 
\end{lemma}
The proof of the lemma is standard, and we therefore defer it to Appendix~\ref{sec:recurrence-solution}. Let us now reason that we can apply the lemma to the sequence $a_k = \phi_{k,k}$. Clearly $0\le a_k\le 1$ as it is a probability, and the fact that the sequence is non-increasing comes from~\eqref{eq:phi-monotonicity-ssm}. 
Let $R = R(\epsilon_0)= R(C_\star,d)$ be sufficiently large that $e^{-R/C_\star}$ is less than $\epsilon_0/3$ and make sure $n_0$ is at least $2R$. Then, let $t_0(R)$ be given by 
$$\max_{\substack{G: |V(G)|\le (2R)^d \\ |E(G)|\le d(2R)^d}} \tmix(X_{t,G})\log(3/\epsilon_0)\,,$$
where the maximum runs over all possible multi-graphs on at most $(2R)^d$ many vertices with at most $d(2R)^d$ many edges. (This captures all the possible boundary conditions induced on $B_{r,e}$.) Evidently this is a finite set, and thus the maximal mixing time above is some constant $T(p,q,R,d) = T(C_\star,p,q,d)$. Now by sub-multiplicativity in time of the total-variation distance between a worst-pair of initializations, which is at most twice the total-variation distance and which in turn bounds $\max_{\omega_0}d_\tv(\omega_0,t)$ in any Markov chain (see e.g.,~\cite{LP}), this implies that
\begin{align*}
    a_{R\vee t_0(R)}\le\phi_{R,t_0(R)}& \le \max_{e\in E(\Lambda_n)} \Big( \max_{\omega_0} \|\mathbb P(X_{t_0,B_{r,e}^\one}^{\omega_0}\in \cdot) - \pi_{B_{r,e}^\one}\|_\tv +  \max_{\omega_0} \|\mathbb P(X_{t_0,B_{r,e}^\zero}^{\omega_0}\in \cdot)  - \pi_{B_{r,e}^\zero}\|_\tv \\
    & \qquad \qquad \qquad  + \|\pi_{B_{r,e}^\one}(\omega_{e}\in \cdot) - \pi_{B_{r,e}^\zero}(\omega_{e}\in \cdot)\|_\tv\Big) \,,
\end{align*}
is at most $\epsilon_0$. Therefore, taking $k_0 = \max\{R,t_0(R)\}$, which clearly only depends on $C_\star,p,q,d$, we have $a_{k_0} \le \epsilon_0$. Further, the recurrence relation is satisfied for all $r\le k \le n/2$ by Proposition~\ref{prop:phi-recurrence-ssm} whenever the random-cluster model on $\Lambda_n^\xi$ satisfies SSM with constant $C_\star$. Therefore, applying Lemma~\ref{lem:recurrence-solution}, we see that there exists $C(C_\star,p,q,d)$ such that $a_k \le Ce^{ - k/C}$ for all $k\le n/2$. Finally the fact that Corollary~\ref{cor:phi-exp-decay-ssm} is for real values of $t$, not just integer, is taken care of by the monotonicity in time~\eqref{eq:phi-monotonicity-ssm} and a change in the constant $C$ in the bound. 
\end{proof}

\subsection{Fast mixing with boundary conditions under SSM}
We now use Corollary~\ref{cor:phi-exp-decay-ssm} to conclude the proof of Theorem~\ref{thm:intro2}. 

\begin{proof}[\textbf{\emph{Proof of Theorem~\ref{thm:intro2}}}]
Consider the grand coupling of the FK dynamics $(X_t^{\omega_0}) = (X_{t,\Lambda_n^\xi}^{\omega_0})$ on $\Lambda_n^\xi$. Let $C_0(C_\star,p,q,d)$ be the constant obtained from Corollary~\ref{cor:phi-exp-decay-ssm}, and let $n_0(C_\star,p,q,d)$ be such that $2d C_0 \log n \le n/2$ and $C_0 n^{-2d} |E(\Lambda_n)| \le 1/4$ for all $n\ge n_0$. 

By the coupling definition of total variation distance and a union bound, it suffices to show that the following is at most $1/4$ when $t = 2d C_0 \log n$: 
\begin{align*}
    \mathbb P\big(X_{t}^\one \ne X_t^\zero\big) \le \sum_{e\in E(\Lambda_n)} \mathbb P\big(X_t^\one(e)\ne X_t^\zero(e)\big)\,.
\end{align*}
By monotonicity, this is at most 
\begin{align*}
    \sum_{e\in E(\Lambda_n)} \mathbb P\big(X_{t,B_{t,e}^\one}^\one(e) \ne X_{t,B_{t,e}^{\zero}}^\zero(e)\big) = \sum_{e\in E(\Lambda_n)} \phi_{t,t}(e) \le |E(\Lambda_n)| \phi_{t,t}\,.
\end{align*}
Applying our bound on $\phi_{t,t}$ from Corollary~\ref{cor:phi-exp-decay-ssm}, and the choice of $t, n_0$, we see that as long as $n\ge n_0$, under the grand coupling
\begin{align*}
    \mathbb P(X_t^\one \ne X_t^\zero) \le C_0 |E(\Lambda_n)|  e^{ - t/C_0} = C_0 n^{-2d} |E(\Lambda_n)|\le 1/4\,,
\end{align*}
concluding the proof. 
\end{proof}

\subsection{Fast mixing on the torus under WSM}\label{subsec:wsm-torus}
We conclude with brief comments describing the differences (actually, simplifications) of the above argument that recover Theorem~\ref{thm:intro1}. 

\begin{proof}[\textbf{\emph{Proof of Theorem~\ref{thm:intro1}}}]
    The proof can be derived essentially by a simpler version of the proof of Theorem~\ref{thm:intro2}, so we just explain the necessary modifications. Rather than defining $B_{m,e}^\one$ and $B_{m,e}^\zero$ as boxes
    that stop at the boundary of the box~$\Lambda_n$, consider them as balls on the torus $\bbT_n$ with wired/free boundaries, respectively. By transitivity of the torus, there is no need to maximize $\phi_{m,t}$ over $e$ ($\phi_{m,t}(e)$ is the same for all~$e$), and it therefore suffices to consider an edge~$e$ incident to the origin. In the proof, since $B_{r,e}^\one, B_{r,e}^\zero$ do not reach the boundary  of~$\Lambda_n$ as $r<n/2$, the bounds in the proof all go through as before, and the assumption of WSM is sufficient to establish~\eqref{eq:2.2temp1-ssm}. 
\end{proof}

\section{Mixing time of the phase-restricted chain under WSM within a phase}\label{sec:mixing-within-a-phase}
We next turn to controlling the mixing time from the all-wired and all-free initializations (and mixtures thereof), rather than from a worst-case initialization. The key point will be that, unlike the worst-case mixing time, the mixing time when initialized in the dominant phase will not blow up as $p\to p_c(q,d)$.

To make this precise, we must formalize what we mean by the \emph{wired} and \emph{free} phases. We use the notation $\widehat \Omega$ and $\widecheck \Omega$ for these respectively, and define them as follows. We emphasize that we have some flexibility with these definitions and they are not as clearly dictated as in the Ising case in~\cite{GhSi21}, where they were determined by the signature of the majority spin. For the random-cluster model, for every $p,q,d$ we define 
\begin{align}
	\widehat \Omega & = \{\omega \subset  E(\bbT_n) : |\cC_1(\omega)|\ge \epsilon n^d\} \label{eq:def:hatOmega}\,; \\ 
	\widecheck \Omega & = \{\omega \subset E(\bbT_n) : |\cC_1(\omega)| \le \epsilon n^d \}\,, \label{eq:def:checkOmega}
\end{align}
where $\cC_1(\omega)$ denotes the largest connected component of~$\omega$. 
Here we choose $\epsilon(p,q,d)$ to be a sufficiently small constant (independent of~$n$) such that
\eqref{eq:def:hatOmega} and~\eqref{eq:def:checkOmega} indeed capture the \emph{wired} and \emph{free} phases of the measure (in the sense made precise following  Definition~\ref{def:exponentially-stable} below). In all the parameter regimes we care about, such a choice of $\epsilon$ exists.

For ease of notation, in what follows we will use $\pi_n$ to denote the random-cluster model on the torus $\pi_{\bbT_n}$.  Denote the measure on $\bbT_n$ conditioned on the wired and free phases as follows: 
\begin{align*}
	\widehat \pi_n = \pi_{n}(\,\cdot \mid \widehat \Omega) \qquad \mbox{and} \qquad \widecheck \pi_n =  \pi_{n}(\,\cdot \mid \widecheck \Omega)\,. 
\end{align*}

Throughout this section, we will take $\biota$ to be a placeholder for either of $\one$ or $\zero$, depending on which of the two phases we are interested in. Our fundamental spatial mixing assumption, as indicated in the introduction, is called \emph{WSM within a phase}. We formalize this notion as follows.  

\begin{definition}\label{def:WSM-within-a-phase}
    We say that WSM holds \emph{within the wired phase} (with constant $C$) if for all $n$ and all $r\le n/2$,
    \begin{align}\label{eq:WSM-within-phase-wired}
        \|\pi_{\Lambda_r^\one} (\omega(\Lambda_{r/2})\in \cdot) - \widehat \pi_{n}(\omega(\Lambda_{r/2})\in \cdot)\|_\tv\le Ce^{ - r/C}\,.
    \end{align}
    Analogously, we say that WSM holds \emph{within the free phase} (with constant $C$) if for all $n$ and all $r\le n/2$,
        \begin{align}\label{eq:WSM-within-phase-free}
        \|\pi_{\Lambda_r^\zero} (\omega(\Lambda_{r/2})\in \cdot) - \widecheck \pi_{n}(\omega(\Lambda_{r/2})\in \cdot)\|_\tv\le Ce^{ - r/C}\,.
    \end{align}
\end{definition}
In words, WSM within the wired (resp., free) phase says that the measure obtained on a box with wired (resp., free) boundary condition is close in total variation in its bulk to the measure on the torus conditioned on having a giant component (resp., not having a giant component). 

Besides WSM within a phase, we make a further assumption that the $\biota$-phase is indeed thermodynamically stable, capturing the fact that the $\epsilon$ we chose in the definitions of~\eqref{eq:def:hatOmega}--\eqref{eq:def:checkOmega} is a good one. 

\begin{definition}\label{def:exponentially-stable}
    Let $\partial \widehat \Omega$ be the set of configurations in $\widehat \Omega$ that are one edge-flip away from $\widehat \Omega^c$.
    We say the wired phase is \emph{exponentially stable} with constant $C_1$ if the following holds: 
    \begin{align}\label{eq:exponentially-stable}
        \pi_{n}(\partial \widehat \Omega \mid \widehat \Omega) \le C_1 e^{ - n^{d-1}/C_1}\,.
    \end{align}
    Similarly, the free phase is \emph{exponentially stable} with constant $C_1$ if \eqref{eq:exponentially-stable} holds with $\widehat\Omega$ replaced by~$\widecheck \Omega$. 
\end{definition}

It is widely expected (see e.g.,~\cite[Conjecture 5.103 and Theorem 5.104]{Grimmett}) that there is a choice of constant $\epsilon(p,q,d)>0$ in \eqref{eq:def:hatOmega}--\eqref{eq:def:checkOmega} such that when $p<p_c(q,d)$ for all $d,q$ the free phase is exponentially stable, when $p>p_c(q,d)$ for all $d,q,$ the wired phase is exponentially stable, and when $q>q_c(d)$ so that the phase transition is discontinuous, both phases are exponentially stable at $p=p_c(q,d)$. It follows from known results that these properties hold in $d=2$ for all $q>4$ and when $q\ge q_0(d)$ in general~$d$ (see Lemma~\ref{lem:exponential-stability-proof} for a short proof). We thus consider this a \emph{milder} condition than WSM within a phase. 

\subsection{Main reduction}
Theorem~\ref{thm:intro3} essentially says that we can reduce the mixing time on the torus $\bbT_n = (\mathbb Z/n\mathbb Z)^d$ 
of the restricted chain $\widehat X^\one_t$ to the mixing time at local $O(\log n)$ scales with wired boundary condition (and similarly for the free phase). 
We state the following mixing time assumption at the local scale that is slightly weaker than a (worst-case) mixing time bound. Suppose there exists a non-decreasing sequence $f(m)$ such that for all $m = O(\log n)$ and all $t>0$, 
\begin{align}\label{eq:assumption-scale-m}
    \max_{e\in E(\bbT_n)}\|\mathbb P\big(X_{t,B_{m,e}^{\biota}}^{\biota}(e) \in \cdot)  - \pi_{B_{m,e}^{\biota}}(\omega_e \in \cdot)\|_\tv \le e^{ - t/f(m)}\,.
\end{align}
Of course by transitivity of the torus the left-hand side is independent of $e$, so we may drop the maximum and just consider a box of side-length $m$ centered at some fixed $e\in E(\bbT_n)$ with $\biota$ boundary condition.

A (worst-case) mixing time (or inverse spectral gap) bound of $f(m)$ for the FK dynamics on $\Lambda_m^{\biota}$ would automatically yield~\eqref{eq:assumption-scale-m}, but we use the above formulation in case one can obtain better bounds on the exponential rate of relaxation when initialized from the $\biota$ configuration. 

Given such a sequence $f(m)$, fix a constant $K$ sufficiently large, and define the following function: 
\begin{align}\label{eq:g(t)}
    g_n(t) = \max\Big\{m \le n : mf(m) \le t\wedge  e^{ n^{d-1}/K}\Big\}\,.
\end{align}

The following theorem is a more formal restatement (in continuous time) of  Theorem~\ref{thm:intro3}. Recall the definition and notation of the FK dynamics restricted to a chain from Section~\ref{subsec:restricted-fk-dynamics}. 
\par\medskip\noindent
{\bf Theorem~\ref{thm:intro3} [formal]}
{\it     Suppose that $(p,q,d)$ is such that the wired phase is exponentially stable with constant $C_1$, WSM within the wired phase holds with constant $C_2$, and~\eqref{eq:assumption-scale-m} holds with $\biota = \one$ for a non-decreasing sequence $f(m)$. Then there exist $C_0,K_0$ depending only on $C_1,C_2$ such that if $g_n(t)$ is as in~\eqref{eq:g(t)} for $K= K_0$, then for every $t\ge 0$, 
    \begin{align*}
        \| \mathbb P(\widehat X_{\bbT_n,t}^\one\in \cdot)  - \widehat \pi_n\|_\tv \le C_0 n^d \exp( - g_n(t)/C_0)\,.
    \end{align*}
    Likewise, if the free phase is exponentially stable with constant $C_1$, WSM within the free phase holds with constant $C_2$, and~\eqref{eq:assumption-scale-m} holds with $\biota = \zero$ for a non-decreasing sequence $f(m)$, then for every $t\ge 0$,} 
        \begin{align*}
        \| \mathbb P(\widecheck X_{\bbT_n,t}^\zero\in \cdot)  - \widecheck \pi_n \|_\tv \le C_0 n^d \exp( - g_n(t)/C_0)\,.
    \end{align*}

\begin{remark}In order to see that this gives the bound of Theorem~\ref{thm:intro3}, notice that in order for the right-hand sides to be $o(1)$, we need $g_n(t)$ to be at least $C\log n$ for a large enough $C$. By the definition of $g_n(t)$ from~\eqref{eq:g(t)}, we see that this will happen if $t \ge C\log n \cdot f(C\log n)$. Taking $f(m)$ to be the  worst-case mixing time on
$\Lambda_m^\one$ then gives the claimed mixing time of $O(\log N\cdot  \tmix(\Lambda_{C\log n}^\one)$ for the restricted chain. (The extra factor of $N$ in Theorem~\ref{thm:intro3} comes from the switch between discrete- and continuous-time dynamics.)
\end{remark}

With the definitions of the phases of the random-cluster model, and the random-cluster notions of WSM within a phase and exponential stability in hand, our proof of Theorem~\ref{thm:intro3} proceeds similarly to the proof of~\cite[Theorem 3.2]{GhSi21}.
For the remainder of this section, let $X_t^{x_0}$ denote $X_{t,\bbT_n}^{x_0}$; similarly, we will drop the $\bbT_n$ subscript in other places where the domain is understood to be the torus. Also, we prove everything for the wired phase $\biota = \one$; the proof for the free phase is completely analogous.

\subsection{Single-edge relaxation within a phase}
We begin by proving a rate of relaxation bound on the single-edge marginals of $\widehat X_t^\one$ to $\widehat \pi_n$ (and similarly of $\widecheck X_t^\zero$ to $\widecheck \pi_n$). 

\begin{proposition}\label{prop:single-site-relaxation}
Suppose the wired phase is exponentially stable with constant $C_1$, WSM within the wired phase holds with constant $C_2$, and~\eqref{eq:assumption-scale-m} holds with $\biota = \one$ for some non-decreasing sequence $f(m)$. Then there exist $C_0,K_0$ depending only on $C_1,C_2$ such that if $g_n(t)$ is as in~\eqref{eq:g(t)} for $K= K_0$, then for every $e\in E(\bbT_n)$, for all $t\le e^{n^{d-1}/K}$, 
\begin{align*}
    |\mathbb P(\widehat X_{t}^{\one}(e) = 1) - \widehat \pi_n(\omega_e = 1)| \le C_0 e^{ - g_n(t)/C_0}\,.
\end{align*}
\end{proposition}

The first step entails circumventing the absence of monotonicity of the FK dynamics when it is restricted to a phase, using the fact that the hitting time of the boundary $\partial \widehat \Omega$ is typically exponentially long if the wired phase is exponentially stable. 
Denote by $\widehat \tau^{x_0}$ the hitting time of $\partial \widehat \Omega$ of the restricted dynamics $\widehat X_t^{x_0}$. Let $\widehat \tau^{\widehat \pi}$ denote this hitting time when one first draws $x_0\sim \widehat\pi_n$ and then considers $\widehat \tau^{x_0}$. Similarly let $X_{t}^{\widehat \pi}$ denote the distribution of $X_t^{x_0}$ when $x_0$ is first drawn randomly from $\widehat \pi$.

\begin{lemma}\label{lem:monotonicity-relations}
Suppose $(p,q,d)$ is such that the wired phase is exponentially stable with constant $C_1$. There exists $C(C_1)$ such that for every $t\ge 0$, we have $\mathbb P ( \widehat \tau^\one\le t) \le \mathbb P(\widehat \tau^{\widehat \pi}\le t) \le C(t\vee 1)e^{ - n^{d-1}/C}$.
\end{lemma}

\begin{proof}
We begin by observing, by monotonicity of the grand coupling and the definition of the restricted dynamics, that
\begin{align}\label{eq:monotonicity-relations}
\widehat X_t^{\widehat \pi} = X_t^{\widehat \pi} \le X_t^\one = \widehat X_t^\one \qquad \mbox{for $0\le t\le \widehat \tau^{\widehat \pi}$}\,.
\end{align}
In particular, this implies the first inequality in the lemma. To obtain the second inequality, fix a sequence of times $t_1,t_2,...$ at which \emph{some} Poisson clock rings for the continuous-time FK dynamics in $\bbT_n$; this process is then a Poisson clock of rate $|E(\bbT_n)| = O(n^d)$. Conditional on any sequence of clock rings, the chain $\widehat X_t^{\widehat \pi}$ is stationary, and thus distributed as $\widehat \pi_n$ for all $t_i$. Then by a union bound,
\begin{align*}
    \mathbb P(\widehat \tau^{\widehat \pi} \le t) & \le \mathbb P\big(\mbox{Pois}(|E(\bbT_n)|\ge n^{d-1}|E(\bbT_n)|(t \vee 1)\big) + \sum_{i\le n^{d-1}|E(\bbT_n)|(t\vee 1)} \mathbb P\big(\widehat X_{t_i}^{\widehat \pi} \in \partial \widehat \Omega\big) \\ 
    & \le Ce^{- n^{d-1}/C} + Cn^{2d-1} (t\vee 1) \widehat\pi_n(\partial \widehat \Omega)\,,
\end{align*}
for a universal constant $C$. We conclude the proof by applying the bound of~\eqref{eq:exponentially-stable} to $\widehat \pi_n(\partial \widehat \Omega)$ and adjusting the constants accordingly. 
\end{proof}

We are now in position to prove Proposition~\ref{prop:single-site-relaxation}. 

\begin{proof}[\textbf{\emph{Proof of Proposition~\ref{prop:single-site-relaxation}}}]
Fix $t\ge 0$ and $e\in E(\bbT_n)$. For ease of notation, define 
\begin{align*}
    \widehat P_t f_e(\one):= \mathbb E[\widehat X_t^\one(e)] - \widehat \pi_n[\omega_e ] = \mathbb P(\widehat X_t^\one(e) = 1) - \widehat \pi_n(\omega_e = 1)\,,
\end{align*}
so that the left-hand side of Proposition~\ref{prop:single-site-relaxation} is $|\widehat P_t f_e(\one)|$.  We begin by lower bounding this quantity, using the fact that under the grand coupling, 
\begin{align*}
    \widehat P_t f_e(\one) = \mathbb E[ \widehat X_t^\one(e) -\widehat X_t^{\widehat \pi}(e)] \ge - \mathbb P(\widehat \tau^{\widehat \pi} \le t)\,.
\end{align*}
Here we used~\eqref{eq:monotonicity-relations} to deduce that the difference of the edge variables is non-negative when $\widehat \tau^{ \widehat \pi} >t$. Thus, by Lemma~\ref{lem:monotonicity-relations}, while $t\le e^{n^{d-1}/K}$ for some $K(C_1)$, we have $\widehat P_t f_e(\one) \ge - C_3 e^{ - n^{d-1}/C_3}$ for some $C_3(C_1)$. 

We now turn to the upper bound. For any $m = m(t)\le n$, by monotonicity of the FK dynamics, 
\begin{align*}
    \widehat P_t f_e(\one) \le \mathbb E[\widehat X_t^\one(e)] - \mathbb E[X_t^\one(e)] + \mathbb E[X_{t,B_{m,e}^\one}^\one(e)] - \widehat \pi_n[\omega_e]\,.
\end{align*}
By~\eqref{eq:monotonicity-relations}, the absolute difference of the first two terms on the right is at most $\mathbb P(\widehat \tau^{\widehat \pi} \le t)$, which is at most $C_3e^{ - n^{d-1}/C_3}$ by Lemma~\ref{lem:monotonicity-relations} for $t\le e^{n^{d-1}/K}$. Using this and a triangle inequality, we get 
\begin{align} \label{eq:prop3.3proof1}
    \widehat P_t f_e(\one) \le C_3 e^{ - n^{d-1}/C_3} + \big|\pi_{B_{m,e}^\one}[\omega_e] - \widehat \pi_n[\omega_e]\big| + \big|\mathbb E[X_{t,B_{m,e}^\one,t}^\one(e)] - \pi_{B_{m,e}^\one,t}[\omega_e]\big|\,.
\end{align}
The second term in~\eqref{eq:prop3.3proof1} is seen to be at most $C_2 e^{ - m/C_2}$ by the assumption of WSM within the wired phase~\eqref{eq:WSM-within-phase-wired}. The third term is exactly the quantity controlled by the assumption in~\eqref{eq:assumption-scale-m}. Combining these, we see that for all $t\le e^{ n^{d-1}/K}$ we have 
\begin{align*}
    \widehat P_t f_e(\one) \le C_3e^{ - n^{d-1}/C_3} + C_2e^{ - m/C_2} + e^{ - t/f(m)}\,.
\end{align*}
Next, taking $m = g_n(t)$, we have that $m\le n$ so that (since $d\ge 2$) the first term is at most $C_3 e^{ - m/C_3}$, and $mf(m) \le t$ so that the third term is at most $e^{ - m}$. Combined, this yields the desired upper bound on $\widehat P_t f_e(\one)$, concluding the proof. 
\end{proof}

\subsection{Fast relaxation within a phase}
We can now prove Theorem~\ref{thm:intro3}. 

\begin{proof}[\textbf{\emph{Proof of Theorem~\ref{thm:intro3}}}]
Consider the total variation distance of interest, 
\begin{align*}
    \|\mathbb P(\widehat X_t^\one\in \cdot) - \widehat \pi_n\|_\tv = \|\mathbb P(\widehat X_t^\one\in \cdot) - \mathbb P(\widehat X_t^{\widehat \pi}\in \cdot) \|_\tv\,.
\end{align*}
By the definition of total variation distance, we have under the grand coupling, 
\begin{align*}
   \|\mathbb P(\widehat X_t^\one\in \cdot) - \widehat \pi\|_\tv \le \mathbb P(\widehat X_t^\one \ne \widehat X_t^{\widehat \pi})
   \le \mathbb P(\widehat X_t^\one \ne \widehat X_t^{\widehat \pi},\widehat \tau^{\widehat \pi} > t) + \mathbb P(\widehat \tau^{\widehat \pi}\le t)\,.
\end{align*}
By Lemma~\ref{lem:monotonicity-relations}, for all $t\le e^{ n^{d-1}/K}$, the second term on the right-hand side is at most $C_3e^{- n^{d-1}/C_3}$. Let us now control the first term. By a union bound, 
\begin{align*}
\mathbb P(\widehat X_t^\one  \ne \widehat X_t^{\widehat \pi}, \widehat \tau^{\widehat \pi} > t)
    & \le \sum_{e\in E(\bbT_n)} \mathbb P(\widehat X_t^\one(e) \ne \widehat X_t^{\widehat \pi}(e), \widehat \tau^{\widehat \pi}>t)  \\ 
    & = \sum_{e\in E(\bbT_n)} \mathbb E[\one\{\widehat X_t^\one(e) \ne \widehat X_t^{\widehat \pi}(e)\}\one\{\widehat \tau^{\widehat \pi}>t\}]\,.
\end{align*}
We can rewrite the indicator of the disagreement as the difference $\widehat X_t^\one(e) - \widehat X_t^{\widehat \pi}$ on the event $\widehat\tau^{\widehat \pi}>t$, to get 
\begin{align*}
      \mathbb P(\widehat X_t^\one  \ne \widehat X_t^{\widehat \pi}, \widehat \tau^{\widehat \pi} > t)  & \le \sum_{e\in E(\bbT_n)} \Big(\mathbb E[\widehat X_t^\one(e)\one\{\widehat \tau^{\widehat \pi}>t\}] - \mathbb E[\widehat X_t^{\widehat \pi}(e)\one\{\widehat \tau^{\widehat \pi}>t\}]\Big) \\
        & \le \sum_{e\in E(\bbT_n)} \Big(\mathbb E[\widehat X_t^\one(e)] - \widehat \pi_n[\omega_e] + \mathbb P(\widehat \tau^{\widehat \pi}\le t)\Big)\,.
\end{align*}
Combining the above and applying Lemma~\ref{lem:monotonicity-relations}, we obtain that for all $t\le e^{n^{d-1}/K}$,
\begin{align*}
    \|\mathbb P(\widehat X_t^\one\in \cdot) - \widehat \pi_n\|_\tv \le |E(\bbT_n)| \max_{e\in E(\bbT_n)} \big(\mathbb E[\widehat X_t^\one(e)] - \widehat \pi_n[\omega_e]\big) + C_3 n^d e^{ - n^{d-1}/C_3}\,.
\end{align*}
The quantity in the parentheses is bounded by Proposition~\ref{prop:single-site-relaxation} as $C_0 e^{-g_n(t)/C_0}$ for all $t\le e^{n^{d-1}/K}$. The second term can be absorbed into this term since $g_n(t)\le n\le n^{d-1}$ for $d\ge 2$. Finally, the constraint on $t$ can be dropped by the fact that total variation distance of a Markov chain to stationarity is non-increasing, and the fact that, by definition, $g_n(t) = g_n(e^{n^{d-1}/K})$ for all $t\ge e^{n^{d-1}/K}$. 
\end{proof}

\section{The large-\texorpdfstring{$q$}{q} cluster expansion of the random-cluster model}\label{sec:q-cluster-expansion}
In this section, we follow the presentation of~\cite{BCT,BCHPT-Potts-all-temp} to introduce the combinatorial framework used to understand the random-cluster model at large~$q$.  We then use this 
framework to deduce certain estimates on uniform connectivity probabilities, 
and later, in Section~\ref{sec:equilibrium-estimates}, use these estimates to deduce WSM within the wired and free phases on the respective sides of the critical point. 
Unless specified otherwise, in this section we will always work in the context of the model
on the torus $\bbT_n = (\mathbb Z/n\mathbb Z)^d$. A slightly delicate point in this section is that we perform all the calculations in the context of $n$ finite and $p$ within $o(1)$ of~$p_c$, rather than exactly at~$p_c$.

For consistency with~\cite{BCHPT-Potts-all-temp,BCT} we reparametrize $p$ by the inverse-temperature parameter $\beta>0$: 
$$\beta(p) := -\log(1-p).$$
Recall that, in the case of integer $q$, this is exactly the parameter transformation that translates the random-cluster measure into the Potts measure via the 
Edwards--Sokal coupling; however, we will use it for all (not necessarily integer)
$q\ge 1$. This naturally gives rise to a critical value $\beta_c(q,d)=\beta(p_c(q,d))$ for all $d\ge 2$ and all $q\ge 1$. 
The critical point of the random-cluster model on $\mathbb Z^d$ has the following asymptotic expansion as $q$ gets large (see, e.g.,~\cite[Theorem 7.34]{Grimmett}): 
\begin{align}\label{eq:beta-c}
    \beta_c(q,d)= {\textstyle\frac{1}{d}} \log q- O(q^{-1/d})\,.
\end{align}
As our aim in this section is to do a large-$q$ cluster expansion, the estimates we are aiming for will hold for all $\beta$ bigger than some threshold which must diverge with $q$. In order to include $\beta_c(q,d)$ in the window we consider for all large $q$, define 
\begin{align}\label{eq:beta-0}
    \beta_h(q,d): = {\textstyle \frac{1}{d}} \log q - 1\,,
\end{align}
and observe that by~\eqref{eq:beta-c}, for $q$ sufficiently large (depending on~$d$), 
we have $\beta_h(q,d)<\beta_c(q,d)$.

\subsection{Combinatorial formalism}
For a random-cluster configuration $\omega$ on~$\bbT_n$, let $V(\omega)$ be the set of vertices belonging to some edge of~$\omega$, and let $H_\omega$ denote the subgraph $(V(\omega),\omega)$. Define the (outer edge) boundary of~$\omega$ by 
\begin{align*}
    \partial\omega = \{e\in E\setminus \omega: e\cap V(\omega) \ne \emptyset\}\,. 
\end{align*}
This boundary is naturally partitioned into $\partial_0 \omega$, where $e\cap V(\omega)$ is a single vertex, and $\partial_1 \omega$, where $e\cap V(\omega)$ is a pair of vertices. 
Using this notation, it is not hard to check that we can write the weight 
$W(\omega):=p^{|\omega|} (1-p)^{|E| - |\omega|} q^{|\rm{Comp}(\omega)|}$ of configuration~$\omega$ in the Gibbs distribution~\eqref{eqn:RCgibbs} as follows:
\begin{align}\label{eq:FK-weights-1}
    W(\omega) = q^{\rm{Comp}(H_\omega)} e^{ - c_{\dis}|V\setminus V(\omega)|} e^{ - c_\ord |V(\omega)|} e^{ - \kappa(|\partial_0 \omega| + 2|\partial_1\omega|)}\,,
\end{align}
where 
\begin{align}\label{eq:FK-reparametrize-beta}
    c_\dis = d\beta - \log q\,;\qquad c_\ord = -d\log(1-e^{-\beta})\,; \qquad \kappa = {\textstyle \frac12} \log(e^\beta - 1)\,.
\end{align}
The expansion we perform will be valid for $\beta > \log 2$ so that $\kappa >0$. Notice that as long as $q$ is large, $\beta>\beta_h(q,d)$ will imply $\beta>\log 2$. The exponential decay in $|\partial_0 \omega|+2|\partial_1\omega|$ will then be the most relevant for us. The set $\partial \omega$ is roughly the \emph{boundary} between the wired and free portions of the configuration, and therefore will play the role of ``contours" between phases near criticality.

To formalize this notion, it will be helpful to view the set $\omega$ as a subset of the continuum torus $\mathbf T_n = (\mathbb R/  n\mathbb R)^d$ of side-length $n$. 
Associate to any configuration $\omega \subset E(\bbT_n)$ a continuum set $\boldsymbol{\omega}$ as follows: 
\begin{enumerate}
    \item For $k=1,...,d$, let $\mathbf {h}_k(\omega)$ be the set of unit $k$-dimensional
    hypercubes in $\mathbf{T}_n$ with corners in $\bbT_n$ all of whose edges are in $\omega$, i.e., $\mathbf{h}_k(\omega)\cap E(\bbT_n)\subseteq \omega$. 
    \item Let $$\boldsymbol{\omega}= \Big\{x\in \mathbf{T}_n : \min_{k=1,...,d} d_\infty(x,\mathbf{h}_k(\omega)) \le 1/4\Big\}\,.$$
\end{enumerate}
Let $\partial \boldsymbol{\omega}$ denote the boundary of $\boldsymbol\omega \subset \mathbf T_n$.

\subsection{The contour representation of a random-cluster configuration}
It will be important to express the random-cluster distribution as a distribution
over collections of ``contours", defined as follows. 

\begin{definition}\label{def:contours}
A \emph{contour} is a (maximal) connected component of $\partial \bomega$. The set of all contours for a given configuration~$\omega$ is denoted~$\Gamma(\omega)$. 
\end{definition}

We can also use $\partial \bomega$ to split the continuum torus $\mathbf T_n$ into two parts, one corresponding to the open components and one to their complement, as follows. 

\begin{definition}Consider $\mathbf T_n \setminus \partial \bomega$. Its connected components are labeled \emph{ordered} if they are a subset of~$\bomega$, and \emph{disordered} if they are a subset of $\mathbf T_n \setminus \bomega$. The union of the ordered components form $\bomega_\ord$ and the union of the disordered components form $\bomega_\dis$. Denote by~$\ell_A$ the label of a connected component~$A$ 
of $\mathbf T_n \setminus \partial \bomega$.
\end{definition}

With this, we obtain the following bijection between configurations and labeled collections of contours. 
\begin{definition}
	A collection of contours $\Gamma$ and a labeling $\ell$ are \emph{admissible} if (i)~$d_\infty(\bgamma_1,\bgamma_2) \ge 1/2$ for every $\bgamma_1, \bgamma_2 \in \Gamma$; and (ii)~$\ell$ assigns to each connected component of $\mathbf T_n \setminus \Gamma$ 
	 one of $\{\dis,\ord\}$ in such a way that adjacent connected components get distinct labels. 
\end{definition}
The map from $\omega$ to $(\Gamma,\ell)$ is then a bijection between configurations and admissible collections of contours with a labeling; see~\cite[Lemma 3.2]{BCHPT-Potts-all-temp}. 

In what follows, for a continuum set $\mathbf A$, let $|\mathbf A|$ denote the number of vertices in $\mathbf A \cap \bbT_n$.  (This notation will not cause any confusion since we will never measure the actual volumes of continuum sets.)  Also, for a contour $\bgamma$, let $\|\bgamma\| = |\bgamma \cap E(\bbT_n)|$ (understood as the cardinality of the discrete set $\bgamma \cap E(\bbT_n)$ with the natural embedding of $E(\bbT_n)$ in $\mathbf{T}_n$). 
Recalling the parameters~\eqref{eq:FK-reparametrize-beta}, we can
rewrite~\eqref{eq:FK-weights-1} as
\begin{align}\label{eq:FK-weights-2}
	W(\omega) = q^{\Comp(\bomega_\ord)} e^{ - c_\dis |\bomega_\dis|} e^{ - c_\ord |\bomega_\ord|} \prod_{\bgamma \in \Gamma(\omega)} e^{-\kappa \|\bgamma\|}\,.
\end{align}
Here, for a continuum set $\mathbf{A}$, we understand $\Comp(\mathbf A)$ to mean the number of connected components of $\mathbf{A}$.  

\subsection{Topologically trivial and non-trivial contours}
Using the continuum embedding of the random-cluster configuration, we now define two types of contours on the torus: topologically non-trivial and topologically trivial ones. The former are contours that ``wrap around" the torus, and are referred to as \emph{interfaces} in~\cite{BCT}. 

\begin{definition}
A contour is \emph{topologically non-trivial} if it has an odd number of intersections with the $i$'th fundamental loop $\{x\in \mathbf{T}_n: x_j \ne 1 \mbox{ for $j\ne i$}\}$ for some $i$.  Otherwise, a contour is called \emph{topologically trivial}. Denote the sets of topologically trivial and non-trivial contours by~$\Gamma_0(\omega)$ 
and~$\Gamma_1(\omega)$, respectively.
\end{definition}

For notational purposes, we use this to partition the random-cluster model's state space into 
\begin{align*}
	\Omega_{\rm tunnel} = \big\{\omega \subset \bbT_n: \Gamma_1(\omega) \ne \emptyset \big\} \qquad \mbox{and}\qquad \Omega_{\rm rest} = \big\{\omega \subset \bbT_n: \Gamma_1(\omega)  = \emptyset\big\}\,.
\end{align*}

The following shows that the existence of topologically non-trivial contours is exponentially unlikely uniformly over \emph{all} $\beta$, which allows us to essentially disregard topologically non-trivial contours in what follows. We obtain this result by combining~\cite[Lemma 6.1 (a)]{BCT} with~\cite[Theorem 1.2]{DCRT19}.

\begin{lemma}\label{lem:tunnel-doesnt-matter}
	For each fixed $d\ge 2$, there exist constants $c>0$ and $q_0$ such that  for all $q\ge q_0$, 
	\begin{align*}
		\pi_n ( \Omega_{\rm{tunnel}}) \le  \exp ( - c (\beta \vee 1) n^{d-1})\,.
	\end{align*}
\end{lemma}

\begin{proof}
If $\beta$ is larger than $\beta_{h}$ from~\eqref{eq:beta-0}, then---as noted in~\cite[lines following (A.1)]{BCT}---(i) and (ii) in Lemma 6.3 of~\cite{BCT} hold, which is all that is used in establishing item~(a) of Lemma 6.1 therein, giving the bound $\exp(-c_1 \beta n^{d-1})$ on the probability of $\Omega_{\rm{tunnel}}$ for all $\beta\ge \beta_h$. To stitch this together with a similar bound for $\beta\le \beta_h$, notice that for every $\beta<\beta_c$, by~\cite[Theorem 1.2]{DCRT19} there exists $c_2(\beta)$ for which one has exponential tails with rate $c_2$ on connected component sizes, and $\Omega_{\rm{tunnel}}$ necessitates a connected component of size at least $\Omega(n^{d-1})$. Since $\beta_h<\beta_c$, we can maximize $c_2$ on the closed interval $\beta\in [0,\beta_h]$ and stitch this together with $c_1$ to obtain the claimed bound for all~$\beta$. 
\end{proof}

\subsection{Geometry of topologically trivial contours}
Our aim is now to express the distribution on $\Omega_{\rm rest}$ as a union of an \emph{ordered phase} $\Omega_\ord$ and a \emph{disordered phase} $\Omega_\dis$, each of which admit convergent polymer expansions in terms of the topologically trivial contours. Abusing notation, from now on, unless otherwise stated, we will use
the unqualified term ``contour" to mean a topologically trivial contour.

It was observed in~\cite[Lemma 4.3]{BCT} that any such contour $\bgamma \in \Gamma_0(\omega)$ has the property that $\mathbf T_n\setminus \bgamma$ consists of exactly two connected components, one of which is an \emph{exterior} $\Ext(\bgamma)$ and the other an \emph{interior} $\Int(\bgamma)$, with the internal one being the one that is simply connected following \cite[Definition 4.4]{BCT}. A contour $\bgamma$ is called \emph{outermost} or \emph{external} if there is no contour $\bgamma'$ such that $\Int(\bgamma) \subset \Int(\bgamma')$. 

Given an admissible contour collection $\Gamma$, there is a single connected component, denoted $\Ext(\Gamma)$, of $\mathbf T_n\setminus \Gamma$ that is \emph{external} in that it is disjoint from $\Int(\bgamma)$ for every $\bgamma\in \Gamma$. In other words, $\Ext(\Gamma) = \bigcap_{\bgamma} \Ext(\bgamma)$. Because of its first characterization, $\Ext(\Gamma)$ receives a label among $\{\dis,\ord\}$ under any admissible labeling of $\Gamma$. Furthermore notice that, given a collection of admissible contours, knowing the label of $\Ext(\Gamma)$ dictates the labels of all connected components of $\mathbf T_n \setminus\Gamma$. Note that a labeling $\ell$ of the connected components of $\mathbf{T}_n \setminus \Gamma$ can also be interpreted as a labeling of contours, where $\ell_\bgamma := \ell_{\Int(\bgamma)}$. (We warn that this is the opposite labeling convention for contours from that used in~\cite{BCT}.)  

We next partition $\Omega_{\rm rest}$ into two sets defined as follows: 
\begin{align*}
    \Omega_\ord = \{\omega\in \Omega_{\rm rest}: \ell_{\Ext(\Gamma(\omega))} = \ord\} \qquad \mbox{and}\qquad \Omega_\dis = \{\omega\in \Omega_{\rm rest}: \ell_{\Ext(\Gamma(\omega))}= \dis\}\,.
\end{align*}
We can then express the partition function of the random-cluster model on $\bbT_n$ restricted to $\Omega_\ord$ as a sum of products as follows. For any continuum set $\mathbf A$, we say a contour $\bgamma$ is in $\mathbf A$ if $d_\infty(\bgamma,\partial \mathbf A) \ge 1/2$, and define $\mathcal G_\ext(\mathbf A)$ to be the set of all \emph{mutually external contour collections} in $\mathbf A$ (meaning all contour collections in $\mathbf A$ in which all the contours are outermost). If we are in $\Omega_\ord$, then the label of all outermost contours must be $\dis$, and vice versa; accordingly, we can divide the set of mutually external contour collections into $\mathcal G_\ext^\ord(\mathbf{A})$ and $\mathcal G_\ext^\dis(\mathbf{A})$. Then, define
\begin{align}
    Z_\ord(\mathbf A) &  := \sum_{\Gamma \in \mathcal G_\ext^\dis(\mathbf A)}  e^{ - c_\ord |\mathbf A \cap \Ext(\Gamma)|} \prod_{\bgamma \in \Gamma} e^{-\kappa \|\bgamma\|} Z_\dis(\Int(\bgamma))\,; \label{eq:Zord-recurse} \\ 
        Z_\dis(\mathbf A) &  := \sum_{\Gamma \in \mathcal G_\ext^\ord(\mathbf A)}  e^{ - c_\dis |\mathbf A \cap \Ext(\Gamma)|} \prod_{\bgamma \in \Gamma} e^{-\kappa \|\bgamma\|} q Z_\ord(\Int(\bgamma))\,. \label{eq:Zdis-recurse}
\end{align}
In the case $\mathbf A = \mathbf T_n$, these give the partition function restricted to $\Omega_\ord$ and $\Omega_\dis$ respectively: 
\begin{align*}
    Z_\dis(\mathbf T_n) = \sum_{\omega \in \Omega_\dis} W(\omega)\,; \qquad q Z_\ord(\mathbf T_n) = \sum_{\omega \in \Omega_\ord} W(\omega)\,.
\end{align*}

Recursing over equations~\eqref{eq:Zord-recurse}--\eqref{eq:Zdis-recurse}, and letting $\mathcal G^\dis(\mathbf A)$ and $\mathcal G^\ord(\mathbf A)$ be the sets of compatible collections of contours in $\mathbf{A}$ all receiving label $\dis$, or $\ord$ respectively, we end up with 
\begin{align}
    Z_\ord(\mathbf A) = e^{ - c_\ord |\mathbf A|} \sum_{\Gamma \in \mathcal G^\dis(\mathbf{A})}\prod_{\bgamma \in \Gamma} K_\ord (\bgamma)\,, \qquad \mbox{where} \qquad K_\ord(\bgamma) = e^{-\kappa\|\bgamma\|} \frac{Z_\dis(\Int(\bgamma))}{Z_\ord(\Int(\bgamma))}\,; \label{eq:Z-ord-K} \\
    Z_\dis(\mathbf A) = e^{ - c_\dis |\mathbf A|} \sum_{\Gamma\in \mathcal G^\ord(\mathbf{A})} \prod_{\bgamma \in \Gamma} K_\dis (\bgamma)\,, \qquad \mbox{where} \qquad K_\dis(\bgamma) = e^{-\kappa\|\bgamma\|} \frac{qZ_\ord(\Int(\bgamma))}{Z_\dis(\Int(\bgamma))}\,. \label{eq:Z-dis-K}
\end{align}
The representation above, and variations on it, will allow us to perform various manipulations of partition functions to establish exponential tails on contour sizes, etc. The key ingredient to proving exponential tails comes from~\cite[Lemma 6.3, items~(i) and~(ii)]{BCT} which we restate here for the reader's convenience. 

In what follows, for $\iota \in \{\dis,\ord\}$, define the free energies 
\begin{align*}
    f_\iota^n = -\frac{1}{n^d} \log Z_\iota(\mathbf{T}_n) \qquad \mbox{and} \qquad f_\iota = \lim_{n\to\infty} f_\iota^n\,.
\end{align*}

\begin{lemma}[{\cite[Lemma 6.3, items (i) and (ii)]{BCT}}]\label{lem:BCT-input-1}
Fix $d\ge 2$. There exist $q_0(d),c>0$ and functions $f_\ord,f_\dis$ such that the following statements hold for $\iota \in \{\dis,\ord\}$, all $q\ge q_0$ and all $\beta\ge \beta_h$. 
\begin{enumerate}[(i)]
\item Let $f = \min \{f_\ord,f_\dis\}$ and let $a_\iota = f_\iota - f$. If $\bgamma$ is such that $\ell_{\Ext(\bgamma)} = \iota$, and $a_\iota \diam(\bgamma)\le c\beta$, then 
\begin{align*}
    K_\iota(\bgamma) \le e^{ - c\beta \|\bgamma\|}\,.
\end{align*}
\item If $\mathbf{A} \subset \mathbf{T}_n$, then 
\begin{align*}
    Z_\iota(\mathbf{A}) \ge e^{ - (f_\ell + e^{ - c\beta n})|\mathbf A|} e^{ - |\partial \mathbf A|}\,,
\end{align*}
and if $\zeta$ is the opposite label of $\iota$, then
\begin{align*}
    Z_\iota(\mathbf A) \le e^{ - (f-e^{ - c\beta n})|\mathbf A|} e^{ 2|\partial \mathbf{A}|} \max_{\Gamma\in \mathcal G_\ext^\zeta(\mathbf A)} e^{ - \frac{a_\ell}{2} |\mathbf A \cap \Ext(\Gamma)|}\prod_{\bgamma\in \Gamma} e^{-\frac{c}{2}\beta \|\bgamma\|}\,.
\end{align*}
\end{enumerate}
\end{lemma}

\subsection{Exponential tails on topologically trivial contours}
Given Lemma~\ref{lem:BCT-input-1}, our aim in what follows is to establish that both ordered measures and disordered measures have exponential tails on their contour lengths,  i.e., $\|\bgamma\|$, even after conditioning on a compatible collection of contours. In what follows, let
$$\pi_\ord := \pi_n(\,\cdot\mid \Omega_\ord) \qquad \hbox{\rm and} \qquad \pi_\dis := \pi_n(\,\cdot \mid \Omega_\dis)\,.$$
We formalize the above statement in the following lemma. 

For a configuration $\omega \in \Omega_{\rm rest}$, let $\Gamma_\ext$ be the random collection of external contours of $\omega$, and for a fixed vertex $v$, let $\bgamma_v$ be the element of $\Gamma_\ext$ nesting $v$. We establish exponential tails on \emph{external} contours, conditional on a collection of mutually external contours.

\begin{lemma}\label{lem:conditional-exp-decay-contours}
For all $\beta \ge \beta_c(q,d) - o(n^{ - 1})$ and $q\ge q_0(d)$, we have for all $v\in \bbT_n$ and all fixed collections $\bar \Gamma_\ext\in \mathcal G_\ext^\dis(\mathbf{T}_n)$, none of which contain $v$ in their interior, 
\begin{align*}
    \pi_\ord(\|\bgamma_v\|\ge r \mid  \bar \Gamma_\ext \subset \Gamma_\ext) \le Ce^{ - \beta r/C}\,.
\end{align*}
For all $\beta \le\beta_c(q,d) + o(n^{-1})$, we have for all $\bar \Gamma_\ext\in \mathcal G_\ext^\ord(\mathbf{T}_n)$, none of which contain $v$ in their interior, 
\begin{align*}
    \pi_\dis (\|\bgamma_v\| \ge r \mid \bar \Gamma_\ext \subset \Gamma_\ext) \le Ce^{ - (\beta\vee 1) r/C}\,.
\end{align*} 
\end{lemma}

\begin{proof}
We first prove the claim for the ordered phase. Fix any collection $\bar \Gamma_\ext$, and fix a $\bar \bgamma$ compatible with $\bar \Gamma_\ext$ such that $v\in \Int(\bar \bgamma)$. Following the derivation of~\eqref{eq:Z-ord-K}, observe that the weight of configurations in $\mathbf{A}$ with ordered external label, having external contour set containing $\bar \Gamma_\ext$, is 
\begin{align}\label{eq:part-function-specific-external-contours}
	Z_{\ord}( \bar \Gamma_\ext \subset \Gamma_\ext) = e^{ - c_\ord |\mathbf{A}|} \sum_{\substack{\Gamma \in \mathcal G_\ext^\dis(\mathbf{A}) \\ \bar \Gamma_\ext \subset \Gamma}} \prod_{\bgamma \in \Gamma} K_\ord(\bgamma)\,.
\end{align}	
With the choice $\mathbf{A} = \mathbf{T}_n$, we wish to consider the ratio 
\begin{align*}
	\frac{Z_\ord(  (\bar \Gamma_\ext\cup \bar \bgamma) \subset \Gamma_\ext)}{Z_\ord(\bar \Gamma_\ext \subset \Gamma_\ext)} = \pi_\ord(\bar\bgamma \in \Gamma_\ext \mid \bar \Gamma_\ext  \subset \Gamma_\ext)\,.
\end{align*}
This is done by defining a map from contour collections $\Gamma$ with $(\bar \bgamma,\bar \Gamma_\ext) \in \Gamma_\ext$ to contour collections $\Gamma$ having $\bar \Gamma_\ext \subset \Gamma_\ext$, that simply deletes the contour $\bar \bgamma$. 
Evidently, for fixed $\bar \bgamma$ this map is injective (the pre-image can be uniquely recovered by adding back $\bar \bgamma$). As a result, we obtain 
\begin{align*}
	\pi_\ord(\bar\bgamma \in \Gamma_\ext \mid \bar \Gamma_\ext  \subset \Gamma_\ext) \le K_\ord (\bar \bgamma)\,.
\end{align*} 
Now recall item (i) of Lemma~\ref{lem:BCT-input-1} which bounds $K_\ord(\bar \bgamma)$ by $e^{ - c\beta \|\bar \bgamma\|}$ when $f_\ord \le f_\dis$. In fact, when $\beta$ is such that  
\begin{align}\label{eq:condition-on-a-ord}
	f_\ord\le f_\dis + o(n^{-1}) \qquad \mbox{so that} \qquad a_\ord \le o(n^{-1})\,,
\end{align}
then we have the bound 
\begin{align*}
	\pi_\ord(\bar\bgamma \in \Gamma_\ext \mid \bar \Gamma_\ext \subset \Gamma_\ext) \le e^{ - c\beta \|\bar \bgamma\|} \qquad \mbox{for all $\bar \bgamma$ in $\mathbf{T}_n$}\,.
\end{align*}
Here we used the fact that the maximal diameter of $\bar \bgamma$ in~$\mathbf{T}_n$ is~$n$ (the specific notion of diameter in $\mathbf{T}_n$ used here is given by~\cite[Section 5.4]{BCT}). Now, we can sum this bound over all possible $\bar \bgamma$ nesting $v$. Using the fact (see e.g.,~\cite[Item (iii) of Lemma 5.8]{BCT}) that for some constant $K=K(d)$, there are at most $K^r$ many possible $\bar \bgamma$'s nesting a fixed vertex $v$ and having the property $\|\bar \bgamma\| = r$, we have by a union bound 
\begin{align*}
	\pi_\ord(\|\bgamma_v\|\ge r \mid \bar \Gamma_\ext \subset \Gamma_\ext) & \le \sum_{k\ge r} \sum_{\bar \bgamma: \|\bar \bgamma\| = k} \pi_\ord ( \bar \bgamma \in \Gamma_\ext \mid \bar \Gamma_\ext  \subset \Gamma_\ext) \\
	& \le \sum_{k\ge r} K^r e^{ - c \beta r} \le Ce^{ - \beta r/C}\,.
\end{align*}

It remains to show that $\beta \ge \beta_c- o(n^{-1})$ implies the assumption of~\eqref{eq:condition-on-a-ord}. When $\beta \ge \beta_c$, we have $a_\ord = f_\ord - \min\{f_\ord,f_\dis\} = 0$ (essentially by definition of $f_\ord,f_\dis$: see~\cite[item (iii) of Lemma 6.3]{BCT}). 
We claim that $a_\ord$ depends in a uniformly Lipschitz manner on $\beta$. To see this, differentiate the finite-volume free energies $f_\iota^{n}$ for $\iota \in \{\dis,\ord\}$ in~$p$ to get 
\begin{align}\label{eq:derivative-of-free-energy}
  \frac{d}{dp} f_\iota^n&  = - \frac{1}{Z_\iota} \frac{1}{n^d} \sum_{\omega \in \Omega_\iota} \frac{d}{d p} W(\omega) = \pi_\iota \big[(1-p)^{-1} n^{-d}(|E|- p^{-1}|\omega|)\big]\,,
\end{align}
which is non-negative and bounded by $d(1-p)^{-1}$ uniformly in $n$. Taking into account the factor $dp/d\beta = e^{-\beta}$, which is uniformly bounded by~$1$, we find that the derivative $\frac{d}{d\beta} f_\iota^n$ in~$\beta$, evaluated at $\beta_c$,
is at most~$d(1-p)^{-1}$.
The Lipschitz constants of $f_\ord$ and $f_\dis$ are thus bounded by some absolute constant when $\beta \approx \beta_c$, and as a consequence we find that $a_\iota=  o(1/n)$ when $|\beta - \beta_c|  = o(1/n)$. 

The analogous bound for the disordered phase is completely symmetrical for $\beta_h \le \beta \le \beta_c + o(n^{-1})$. To see that it extends to the even higher temperature regime $\beta \in [0,\beta_h]$, we use the results of~\cite{DCRT19}. 
Since $\beta_h<\beta_c$, the role of conditioning on $\Omega_\dis$ is negligible so it suffices to prove the bound under $\pi_n$ itself. In order for there to be a contour $\gamma$ such that $\|\bgamma_v\|\ge r$, there must be a connected component in $\omega$ of diameter at least $r$. Expose the connected components bounded by the contours of $\bar \Gamma$, inducing free boundary conditions on $\Ext(\bar \Gamma)$. We wish to bound the probability of there existing a connected component of size at least $k\ge r$ at distance at most $k$ from $v$ (so that it can nest $v$). The fact that we condition on the contours~$\bar \Gamma$ being external is equivalent to there not being any connected component nesting one of those contours, a decreasing event, so by the FKG inequality we can drop that conditioning and consider the probability of there existing a connected component of size at least $k\ge r$ at distance at most $k$ from $v$ under $\pi_{\Ext(\bar \Gamma)^\zero}$. This decays exponentially in $k$ for every $k$ by a union bound and the exponential tails on connected components from~\cite{DCRT19}.
\end{proof}

Using essentially the same argument, we can also obtain a version of Lemma~\ref{lem:conditional-exp-decay-contours} that applies when the phase is picked by boundary conditions, rather than by conditioning on $\Omega_\ord$ or $\Omega_\dis$. 

\begin{lemma}\label{lem:boudary-condition-exp-decay-contours}
For all $\beta \ge \beta_c(q,d) - o(n^{ - 1})$ and $q\ge q_0(d)$ we have, for all fixed collections $\bar \Gamma_\ext \in \mathcal G_\ext^\dis(\mathbf{\Lambda}_n)$,
\begin{align*}
     \pi_{\Lambda_n^\one} (\|\bgamma_v\|\ge r \mid  \bar \Gamma_\ext  \subset \Gamma_\ext) \le Ce^{ - \beta r/C}\,.
\end{align*}
Similarly, for all $\beta \le\beta_c(q,d) + o(n^{-1})$ we have, for all fixed $\bar \Gamma_\ext \in \mathcal G_\ext^\ord(\mathbf{\Lambda}_n)$\,,
\begin{align*}
     \pi_{\Lambda_n^\zero} (\|\bgamma_v\|\ge r \mid \bar \Gamma_\ext\subset \Gamma_\ext) \le Ce^{ - (\beta\vee 1) r/C}\,.
\end{align*} 
\end{lemma}

\begin{proof}
    We can do a similar expansion of the partition functions of the random-cluster model on $\Lambda_n^\one$ and $\Lambda_n^\zero$ as done on the torus to arrive at~\eqref{eq:Zord-recurse}--\eqref{eq:Zdis-recurse}. One indirect way to do so is to naturally embed $\Lambda_n$ in $\bbT_{3n}$, and view the boundary conditions as coming either from the all-wired or all-free configurations on $E(\bbT_{3n}\setminus \Lambda_n)$. For instance, the contour given by $\partial \mathbf{\Lambda}_n$, with (internal) label $\ord$, would necessarily wire all edges along $\partial \Lambda_n$, which is the same as that model with wired boundary conditions up to a uniform proportionality factor. 
    
    This was the argument carried out in~\cite[Proposition 3.14]{BCHPT-Potts-all-temp}. If $Z_{\Lambda_n^\one}$ and $Z_{\Lambda_n^\zero}$ denote the partition functions associated with $\pi_{\Lambda_n^\one}$ and $\pi_{\Lambda_n^\zero}$ respectively, then it was shown there that there exist contours $\bgamma_\zero^{n}$, $\bgamma_\one^{n} \subset \mathbf{T}_{3n}$ such that 
    \begin{align}\label{eq:part-function-bc-conditioning}
        Z_{\Lambda_n^\one} = q  p^{|E(\Lambda_n)| - d|\Int(\bgamma_\one^n)|} Z_\ord(\Int(\bgamma_\one^n))   \qquad \mbox{and}\qquad Z_{\Lambda_n^\zero} = (1-p)^{\|\bgamma_\zero^{n}\|/2} Z_\dis(\Int(\bgamma_\zero^{n}))
         \,.
    \end{align}
    At this point, recall that~\eqref{eq:part-function-specific-external-contours} applied for general~$\mathbf{A}$, so we can take $\mathbf{A} =\Int(\bgamma_\one^n)$. From this starting point, one establishes that, for every fixed $\bar \bgamma$ and fixed $\bar \Gamma_\ext$, 
    \begin{align*}
        \pi_{\Lambda_n^\one}(\bar \bgamma \in \Gamma_\ext \mid \bar \Gamma_\ext \subset \Gamma_\ext ) \le K_\ord(\bar \bgamma)
    \end{align*}
    exactly as in the previous proof, noting that the extra coefficients in~\eqref{eq:part-function-bc-conditioning} cancel out between the numerator and denominator when passing to probabilities. The remainder of the proof is unchanged.
\end{proof}

\subsection{Equivalence of different notions of phases}
We now use Lemma~\ref{lem:conditional-exp-decay-contours} to establish the following equivalence between the restrictions of the measure~$\pi_n$ to $\Omega_\ord$ and to $\widehat \Omega$ when $q\ge q_0$. We emphasize that the reason we presented the results in Section~\ref{sec:mixing-within-a-phase} for $\widehat \Omega$ and $\widecheck \Omega$ is that when $q$ is intermediate, i.e., larger than $q_c(d)$ so there is phase coexistence at $p_c$ but not very large, one does not expect Lemma~\ref{lem:conditional-exp-decay-contours} to hold, but still the coexistence of $\widehat \Omega$ and $\widecheck \Omega$ should hold.

In what follows, we use $A \triangle B$ to denote the symmetric difference between sets~$A$ and~$B$. 

\begin{lemma}\label{lem:equivalence-of-ord-and-hat}
   Fix $d\ge 2$ and $q\ge q_0$. Let $\epsilon<\frac{1}{2}$ be any value used in the definition of $\widehat \Omega$. There exists $c(\epsilon)>0$ such that uniformly over all $\beta \ge \beta_c(q,d) - o(n^{-1})$, 
    \begin{align*}
        \pi_n(\Omega_\ord \triangle \widehat \Omega)\le Ce^{ - c\beta n^{d-1}} \qquad \mbox{and}  \quad \|\pi_\ord - \widehat \pi_n\|_\tv\le Ce^{ - c\beta n^{d-1}}\,.
    \end{align*}
    An analogous statement holds for $\Omega_\dis$ and $\widecheck \Omega$ at all $\beta\le \beta_c(q,d) + o(n^{-1})$.  
\end{lemma}

\begin{proof} 
Observe that 
\begin{align*}
    \pi_n(\widehat \Omega \setminus \Omega_\ord) \le \pi_n(\Omega_{\rm tunnel})  + \pi_n(\widehat \Omega \cap \Omega_\dis)\,.
\end{align*}
The first of these two terms is at most $Ce^{ - c\beta n^{d-1}}$ by Lemma~\ref{lem:tunnel-doesnt-matter}. For the second term, if $\beta \ge \beta_c + \frac{c(\delta)}{n}$ is such that $a_\dis > \frac{\delta \beta}{n}$, for a $\delta$ to be chosen later, then a combination of the bounds in item (ii) of Lemma~\ref{lem:BCT-input-1} yields 
\begin{align*}
    \pi_n(\Omega_\dis)\le \frac{Z_\dis(\mathbf{T}_n)}{Z_\ord(\mathbf{T}_n)}\le e^{ - a_\dis |\mathbf{T}_n|} e^{e^{-c\beta n}|\mathbf{T}_n|} \le e^{ - \delta \beta n^{d-1}/2}
\end{align*}
when $n$ is large enough. On the other hand, when $|\beta - \beta_c|\le c(\delta)/n$ we use the fact that $\widehat \Omega \cap \Omega_\dis$ requires that there exists a contour $\bgamma \in \Gamma$ having $|\Int(\bgamma)| \ge \epsilon n^d$. As long as $\delta <1$ in the choice of $\beta$ above, the exponential tails on $\|\bgamma\|$ in Lemma~\ref{lem:conditional-exp-decay-contours} apply because~\eqref{eq:condition-on-a-ord} is satisfied, and we get by a union bound 
\begin{align*}
    \pi_n(\widehat\Omega \cap \Omega_\dis)\le \pi_{\dis}(\exists v: |\Int(\bgamma_v)|\ge \epsilon n^d) \le \pi_\dis(\exists v: \|\bgamma_v\|\ge \epsilon n^{d-1})\le  Ce^{ - \beta \epsilon n^{d-1}/C}\,.
\end{align*}

We now turn to bounding $\pi_n(\Omega_\ord \setminus \widehat \Omega)$. In order for $\Omega_\ord\setminus \widehat \Omega$ to occur, it must be that $|\Ext(\Gamma)|\le \epsilon n^d$ since $\Ext(\Gamma)$ is itself a connected component whenever $\omega \in \Omega_\ord$. This event was shown to have probability at most $e^{ - c\beta n^{d-1}}$ by~\cite[Lemma 6.2]{BCT} uniformly over $\beta \ge \beta_c$. To extend the bound to $|\beta - \beta_c|\le n^{-1}$, we use the fact that the Radon--Nikodym derivative between $\pi_{n,p}$ and $\pi_{n,p_c}$ is bounded as 
\begin{align}\label{eq:radon-nikodym-p}
   \Big\|\frac{\pi_{n,p}}{\pi_{n,p_c}}\vee \frac{\pi_{n,p_c}}{\pi_{n,p}}\Big\|_\infty \le (1+|p-p_c|)^{|E|}\,,
\end{align}
which is clearly $\exp(o(n^{d-1}))$ if $|p-p_c| = o(n^{-1})$, or equivalently if $|\beta  - \beta_c|= o(n^{-1})$. 

The total variation bound on the conditional distributions follows from the fact that, for any event $A$, 
\begin{align*}
    \Big|\frac{\pi_n(A\cap \Omega_\ord)}{\pi_n(\Omega_\ord)} - \frac{\pi_n(A \cap \widehat \Omega)}{\pi_n(\widehat \Omega)}\Big| & \le \frac{1}{\pi_n(\Omega_\ord)}|\pi_n(A\cap \Omega_\ord) - \pi_n(A\cap \widehat\Omega)| \\
    & \qquad + \frac{\pi_n(A \cap \widehat \Omega)}{\pi_n(\widehat \Omega)\pi_n(\Omega_\ord)}|\pi_n(\widehat \Omega) - \pi_n(\Omega_\ord)| \\ 
    & \le \frac{2}{\pi_n(\Omega_\ord)}\cdot \pi_n(\Omega_\ord\triangle \widehat\Omega)\,.
\end{align*}
Using the fact that $\pi_n(\Omega_\ord)$ is uniformly bounded away from zero for all $\beta \ge \beta_c$ (see~\cite[Eq.~(6.7)]{BCT}), and the Radon--Nikodym bound of~\eqref{eq:radon-nikodym-p}, the first term above is $\exp(o(n^{d-1}))$ uniformly over $\beta \ge \beta_c- o(n^{-1})$, and the second symmetric difference term is what we already bounded above. 

The proof for $\Omega_\dis$ and $\widecheck \Omega$ is symmetric for $\beta_h \le \beta\le \beta_c+o(1/n)$. To extend it to $\beta\in [0,\beta_h]$, notice that each of the quantities we would bound in the symmetric proof for the high-temperature regime in the above argument is governed by events that involve the existence of a connected component of size at least $n^{d-1}$, for which we can use the bound of~\cite{DCRT19}. 
\end{proof}

\subsection{Exponential stability of the wired and free phases}
The aim in this section is to establish exponential stability, as defined in Definition~\ref{def:exponentially-stable}, for the wired and free phases in the right parameter regimes, when $q$ is sufficiently large. 

\begin{lemma}\label{lem:exponential-stability-proof}
    Fix $d\ge 2$ and $q\ge q_0(d)$. For any $\epsilon<\frac{1}{2}$, there exists $C_1(\epsilon)$ such that for all $\beta\ge \beta_c - o(n^{-1})$  
    \begin{align*}
        \widehat \pi_n (\partial \widehat \Omega) \le C_1 e^{ - n^{d-1}/C_1}\,.
    \end{align*}
    The symmetric statement holds for $\widecheck \pi_n(\partial \widecheck \Omega)$ for all $\beta \le \beta_c + o(n^{-1})$. 
\end{lemma}

\begin{proof}

    By~\cite[Eq.~(6.7)]{BCT}, we have that $\pi_n(\Omega_\ord)\ge q/(q+1)-o(1)$ when $p \ge p_c(q,d)$. More generally using Lemma~\ref{lem:equivalence-of-ord-and-hat} to compare $\widehat \Omega$ to $\Omega_\ord$, and the Radon--Nikodym bound of~\eqref{eq:radon-nikodym-p}, we have 
    $$\pi_n(\widehat \Omega) \ge \exp( - o(n^{d-1}))\,,$$
    uniformly over all $\beta \ge \beta_c - o(n^{-1})$. Next consider the mass of $\partial \widehat \Omega$ which consists of the set of configurations which are one edge-flip away from $\widehat \Omega^c$. By Lemma~\ref{lem:equivalence-of-ord-and-hat}, and the fact that $\pi_n(\Omega_\ord)\ge \exp(- o(n^{d-1}))$ uniformly over $\beta \ge\beta_c - o(n^{-1})$, it suffices to provide an $\exp(-\Omega(n^{d-1}))$ bound for some $\delta>0$ on
    \begin{align}\label{eq:exp-stability-suffices-to-show}
        \pi_\ord (\partial \widehat \Omega) \le  \pi_\ord\big(|\Ext(\Gamma)|\le (\tfrac{1}{2} + \delta)n^d\big) + \pi_\ord\big(\exists v: \|\bgamma_v\| \ge \delta n^{d-1}\big)\,.    \end{align}
    The inequality here can be seen because in order for a configuration to be one spin-flip away from having no component of size larger than $\epsilon n^d$, for any $\epsilon<1/2$, either its external component is smaller than size $1/2 + \delta$, or if not then its external component is its largest component, and it has a cut-edge $e\in \omega$ such that if the state of $e$ is changed to closed, then it splits off a component of size at least $\delta n^{d}$. In order for this latter situation to occur, there must be a contour having size at least $\delta n^{d-1}$. 

    As long as $q$ is sufficiently large, the first of the two terms of~\eqref{eq:exp-stability-suffices-to-show} has probability at most $e^{ - c\beta n^{d-1}}$ by~\cite[Lemma 6.2]{BCT} uniformly over $\beta \ge \beta_c$, extended to $\beta\ge \beta_c - o(n^{-1})$ by the Radon--Nikodym bound we have been using repeatedly above. The second term in~\eqref{eq:exp-stability-suffices-to-show} is bounded by $Ce^{ - c\beta \delta n^{d-1}}$ per Lemma~\ref{lem:conditional-exp-decay-contours} and a union bound over $v$. Altogether, that yields the claimed bound. 
    
    The proof for $\widecheck \pi_n(\partial \widecheck\Omega)$ is similar for $\beta_h\le \beta \le \beta_c + o(1/n)$. It effectively reduces to bounding the probability of there existing \emph{two} connected components whose sizes together sum up to $\epsilon n^d$, which in particular requires there to be at least one contour under $\pi_\dis$ of size $\epsilon n^{d-1}/2$. To extend it to $\beta\in [0,\beta_h]$, one can work directly with the event on connected components in $\omega$, and apply the exponential tails on component sizes from~\cite{DCRT19}. 
\end{proof}

\subsection{The unstable phase near criticality}
In order for us to be able to faithfully sample from the random-cluster distribution near the critical point in finite volumes without knowing the exact value of the critical point, it will be important to have an understanding of how close to the critical point the unstable phase becomes negligible. The following lemma shows that the $o(n^{-1})$ window of temperatures around~$p_c$ through which the above estimates apply is sufficient to cover the window of temperatures at which that phase contributes more than exponentially small mass.   

\begin{lemma}\label{lem:unstable-phase-negligible}
    Fix $d\ge 2$ and $q\ge q_0(d)$ and any $\delta>0$. For every $p\ge p_c + \Omega(n^{-(1+\delta)})$, 
    \begin{align*}
       \pi_n(\Omega_\dis) \le \exp( - c n^{d-1-\delta})\,.
    \end{align*}
    The same bound holds for $\pi_n(\Omega_\ord)$ as soon as $p\le p_c- \Omega(n^{-(1+\delta)})$. 
\end{lemma}
\begin{proof}
    Recall from~\cite{BCT} that $\beta_c$ is the unique point at which $f_\ord = f_\dis$, and when $\beta>\beta_c$ necessarily $f_\dis >f_\ord$. Our aim is to understand the asymptotic dependence of $a_\dis = f_\dis - f_\ord$ on $(\beta - \beta_c)$ as $\beta \to\beta_c$ from above. Recall that $f_\iota = \lim_{L\to\infty} f_\iota^L$ and by the line above (A.8) in~\cite{BCT}, $|f_\iota - f_\iota^L|\le 2e^{  - c\beta L}$. Letting $a_\dis^L = f_\dis^L - f_\ord^L$, it suffices therefore to control (uniformly in $L$) the dependence of $a_\dis^L$ on $(\beta - \beta_c)$. 

    For this purpose, differentiating as in~\eqref{eq:derivative-of-free-energy}, 
    \begin{align*}
        \frac{d}{dp}a_\dis^L = \frac{d}{dp}(f_\dis^L - f_\ord^L) = p^{-1}(1-p)^{-1} \big(\pi_{L,\ord}[L^{-d}|\omega|] - \pi_{L,\dis}[L^{-d}|\omega|]\big)\,,
    \end{align*}
    where $\pi_{L,\ord}$ and $\pi_{L,\dis}$ denote those measures on $\bbT_L$. By~\cite[Lemma 6.2]{BCT}, when $q$ is sufficiently large, for all $\beta \ge \beta_c$, under $\pi_\ord$, $|E(\Ext(\Gamma))| \le |\omega|$ is itself at least $3|E|/4$ with probability $1-o_L(1)$. On the other hand, while $\beta - \beta_c = o(L^{-1})$, $\pi_\dis$ satisfies the exponential tail bound of Lemma~\ref{lem:conditional-exp-decay-contours}, and as a consequence, when $q$ is sufficiently large, the expected fraction of open edges, which is bounded by $d$ times the expected number of vertices having non-empty $\bgamma_v$, is at most $|E|/4$ (say). Combined with the factors $p^{-1}(1-p)^{-1}$ and $\frac{dp}{d\beta}$ we find that, uniformly over $L$, at $\beta_c$, $\frac{d}{d\beta} a_\dis^L$ is at least some constant $c_0(p,q,d)>0$. 
    
    Continuity of $a_\dis^L$ in $\beta$ (as noted in~\cite[Appendix A.4]{BCT}) then implies that, for $\beta>\beta_c$, 
    \begin{align*}
        a_\dis^L(\beta) \ge a_\dis^L(\beta_c) + c_0(\beta-\beta_c) + O((\beta-\beta_c)^2)\,.
    \end{align*}
    These being free energies corresponding to absolutely convergent power series, when we pass to the $L\to \infty$ limit, we get the same inequality for $a_\dis(\beta)$. 
    Now consider a sequence of $\beta$ such that $\beta= \beta_c + n^{-(1+\delta)}$ for a fixed $\delta>0$. At such~$\beta$, there is some $c_1$ such that $a_\dis(\beta) \ge c_1 n^{-1+\delta}$. Fix this $\beta$ and consider the quantity
    \begin{align*}
       \pi_n(\Omega_\dis)\le  \frac{Z_\dis(\mathbf{T}_n)}{Z}\le \frac{Z_\dis(\mathbf{T}_n)}{Z_\ord(\mathbf{T}_n)}\,.
    \end{align*}
    By combining the two parts of item (ii) in Lemma~\ref{lem:BCT-input-1}, as observed in~\cite[Eq.~(30)]{BCHPT-Potts-all-temp}, this satisfies 
    \begin{align*}
        \frac{Z_\dis(\mathbf{T}_n)}{Z_\ord(\mathbf{T}_n)} \le \exp( - \tfrac{1}{4} \min\{a_\dis n^d, c\beta n^{d-1}\})\,.
    \end{align*}
    Applying our bound on $a_\dis$ then implies the desired bound at $\beta = \beta_c +n^{-(1+\delta)}$. The fact that the bound holds for all larger $\beta$ follows from monotonicity of the random-cluster model and the fact that $\Omega_\dis$ is a decreasing event. 
    The bound on $\pi_n(\Omega_\ord)$ follows symmetrically. 
\end{proof}

\section{Spatial mixing properties of the random-cluster model}\label{sec:equilibrium-estimates}
Recall the notions of WSM, SSM, and WSM within a phase from Definitions~\ref{def:WSM}--\ref{def:ssm-per-bc} and~\ref{def:WSM-within-a-phase}. In this section, we will establish reductions from these notions of spatial mixing to certain exponential decay properties on contours with label $\ord$ or~$\dis$. These reductions will make it easier to ascertain the off-critical parameter regimes in which WSM and SSM hold, and will also allow us to establish the WSM within a phase property at large~$q$, leveraging the arguments of Section~\ref{sec:q-cluster-expansion}. In particular, this section also includes proofs of two of our main results on mixing
times, Corollaries~\ref{cor:intro1} and~\ref{cor:intro2} from the introduction.

Since we are interested in all temperature regimes, we need to consider both $\ord$ and $\dis$ connectivities simultaneously. Recall the continuum embedding of a random-cluster configuration $\omega$, denoted $\bomega$ in Section~\ref{sec:q-cluster-expansion}. For any two (continuum) points $x,y$, write 
\begin{align*}
    x\stackrel{\ord}\longleftrightarrow y & := \{\omega: x,y \mbox{ are in the same connected component of }\bomega\}\,; \\ 
    x\stackrel{\dis}\longleftrightarrow y & := \{\omega: x,y\mbox{ are in the same connected component of }\mathbf{\Lambda}\setminus\bomega\}\,.
\end{align*}
Notice that, by construction of $\bomega$, for vertices $v,w\in \mathbb Z^d$, $v\stackrel{\ord}\longleftrightarrow w$ is equivalent to $v$ and $w$ belonging to the same connected component of $\omega$ itself. When $d=2$, for dual vertices $v,w\in (\mathbb Z + \frac 12)^2$, $v\stackrel{\dis}\longleftrightarrow w$ is the same as the usual notion of dual connectivity. In higher dimensions, if one defines $F_*(\omega)$ as the set of unit $(d-1)$-cells that are dual to closed edges,\footnote{A $k$-cell is a $k$-dimensional unit hypercube. A $d-1$-cell is dual to an edge $e$ if they share a barycenter and $e$ is normal to the cell.} then $v\stackrel{\dis}\longleftrightarrow w$ indicates that $v,w\in (\mathbb Z+\frac 12)^d$ are connected through $F_*(\omega)$ (where two dual $(d-1)$-cells are adjacent if they are incident on one another).

\subsection{Weak spatial mixing}\label{subsec:wsm}
The first reduction of this section is from WSM to bulk exponential decay of $\ord$ and $\dis$ connectivities (i.e., connectivities at distance $\Theta(n)$ from the boundary). The proof is quite standard, and such coupling arguments are common in the literature when working with $\ord$-connectivities, but we include it since the later proofs for SSM and WSM within a phase will build on similar ideas.  

\begin{lemma}\label{lem:exp-decay-to-WSM}
    Suppose $(p,q,d)$ are such that one of the following hold: 
    \begin{align}\label{eq:exp-decay-wsm}
    \pi_{\Lambda_m^\one}(\partial \Lambda_{m/2}\stackrel{\ord}\longleftrightarrow \partial \Lambda_m)\le Ce^{ - m/C}\qquad \mbox{or}\qquad \pi_{\Lambda_m^\zero}(\partial \Lambda_{m/2}\stackrel{\dis}\longleftrightarrow\partial \Lambda_m) \le Ce^{ - m/C}\,.
    \end{align}
    Then the random-cluster model satisfies WSM. 
\end{lemma}

\begin{proof}
We start with the implication from the first assumption. We construct a coupling between $\omega^\zero \sim \pi_{\Lambda_n^\zero}$ and $\omega^\one \sim \pi_{\Lambda_n^\one}$ such that they agree on $\Lambda_{n/2}$ on the complement of the event 
    \begin{align}\label{eq:connectivity-event-wsm}
        \big\{\omega^\one: \partial \Lambda_{n/2} \stackrel{\ord}\longleftrightarrow \partial  \Lambda_n\big\}\,.
    \end{align}
    The coupling goes as follows. Expose the connected component of $\partial \Lambda_n$ (i.e., the union of the connected components of all vertices in $\partial \Lambda_n$) under $\omega^\one$ and call that set $D$. Notice that this is exactly the set of vertices in $\Ext(\Gamma_\ext(\omega^\one))$, and $\partial \mathbf{D}= \Gamma_\ext(\omega^\one)$. 
    Expose also the values of $\omega^\zero(D)$ under the grand monotone coupling with $\omega^\one$. 
    Having revealed $D$, we have also revealed that all edges of 
    \begin{align*}
        \partial_\out D = \{e = (v,w): v\in D, w\notin D\}
    \end{align*}
    are closed under $\omega^\one$, which is what gave rise to the contour collection $\partial \mathbf{D}$, all of whose interior labels are $\dis$. 
    Therefore, we can sample the remainder of the configuration of $\omega^\one$ by drawing 
    \begin{align*}
        \omega^\one(E(\Lambda_n\setminus D)) \sim \pi_{\Lambda_n^\one}(\,\cdot \mid \omega^\one(D \cup \partial_\out D) )\,,
    \end{align*}
    which by the domain Markov property is exactly the random-cluster distribution on $\Int(\partial \mathbf{D})$ with free boundary conditions. By monotonicity of the coupling of $(\omega^\zero,\omega^\one)$, all edges of $\omega^\zero(\partial_\out D)$ are also closed, so by the domain Markov property the marginal $$\pi_{\Lambda_n^\zero}(\,\cdot \mid \omega^\zero(D\cup \partial_\out D))$$
    is also just a random-cluster distribution on $\Int(\partial \mathbf{D})$ with free boundary conditions. Therefore, we can conclude the coupling simply by setting $\omega^\zero(E(\Lambda_n\setminus D)) = \omega^\one(E(\Lambda_n\setminus D))$. 
    
    Thus $\omega^\one(E(\Lambda_{n/2})) = \omega^\zero(E(\Lambda_{n/2}))$ as long as $\Lambda_{n/2} \cap D = \emptyset$. As such, 
    \begin{align*}
        \|\pi_{\Lambda_m^\zero}(\omega(\Lambda_{m/2})\in \cdot) - \pi_{\Lambda_m^\one}(\omega(\Lambda_{m/2}\in \cdot)\|_\tv \le \pi_{\Lambda_m^\one}(\Lambda_{n/2} \cap D \ne \emptyset) = \pi_{\Lambda_m^\one}(\partial \Lambda_{m/2}\stackrel{\ord}\longleftrightarrow\partial \Lambda_m)\,,
    \end{align*}
    which is exactly the first quantity bounded in~\eqref{eq:exp-decay-wsm}.
    
    The proof assuming the bound on the disordered connectivity events in~\eqref{eq:exp-decay-wsm} proceeds analogously, by revealing $\Ext(\Gamma_\ext(\omega^\zero))$ using adjacency of the dual $(d-1)$-cells we called $F_*(\omega)$. Then the set of edges whose states are exposed to reveal the connected component of $\partial \Lambda_n$ in $F_*(\omega)$ induce a wired boundary on the interiors of $\Gamma_\ext(\omega^\zero)$ as expected, and the domain Markov property can be applied on that set. 
\end{proof}

\begin{proof}[\textbf{\emph{Proof of Corollary~\ref{cor:intro1}}}]
    By Theorem~\ref{thm:intro1} it suffices to establish that WSM holds in all the parameter regimes claimed in Corollary~\ref{cor:intro1}.  
    In~\cite{DCRT19} it was shown that the random-cluster model has exponential decay of connectivities of the form of the first bound in~\eqref{eq:exp-decay-wsm} whenever $d\ge 2$, $q\ge 1$ and $p<p_c(q,d)$. Together with Lemma~\ref{lem:exp-decay-to-WSM} that implies WSM at all $p<p_c(q,d)$. The fact that the Ising random-cluster model ($q=2$) satisfies WSM when $d\ge 2$ at all $p>p_c(q,d)$ was directly shown in~\cite{DCGR20}. Finally, in order to see that the random-cluster model satisfies WSM whenever $p$ is sufficiently large, one notices that the second item in~\eqref{eq:exp-decay-wsm} holds at $p$ sufficiently large by comparison of the set of closed edges to a sub-critical percolation process on the graph induced by the dual notion of adjacency we considered for closed edges.  
\end{proof}

\begin{remark}It seems plausible that the methods of~\cite{BCT} as outlined in Section~\ref{sec:q-cluster-expansion} are strong enough to show WSM (and in turn SSM on $\Lambda_n^\one$) for all $p>p_c(q,d)$ when $q\ge q_0(d)$ is sufficiently large. The challenge is to establish the second bound of~\eqref{lem:exp-decay-to-WSM} at slightly super-critical~$p$. Under $\pi_{\Lambda_m^\zero}$ and $p>p_c(q,d)$ there will necessarily be a long contour confining almost all of the volume of $\Lambda_m$, so the question reduces to understanding the geometry of this separating surface, which presumably will remain within an $O(1)$ distance from $\partial\Lambda_m$ with exponential tails on oscillations away from that. 
\end{remark}

\subsection{Strong spatial mixing}
The next reduction of this section is between strong spatial mixing and connectivity decay that holds uniformly over vertices in the box. This reduction will only be applicable for a special class of boundary conditions which are called \emph{side-homogeneous}. 

\begin{definition}\label{def:side-homogeneous-bc}
Side-homogeneous boundary conditions are those in which a subset of the $2d$ sides of $\partial \Lambda_m$ are selected and all vertices in those sides are in one single element of the boundary partition; all vertices in the other sides of $\partial \Lambda_m$ are singletons. In particular, the all-wired and all-free boundary conditions are side-homogeneous.  
\end{definition}

\begin{lemma}\label{lem:exp-decay-to-SSM}
    Fix any side-homogeneous boundary condition $\eta$. Suppose $(p,q,d)$ are such that either: 
   \begin{align}\label{eq:exp-decay-ssm}
   \max_{v}\pi_{\Lambda_m^\eta \cap B_{r,v}^\one}(\partial B_{r/2,v}\stackrel{\ord}\longleftrightarrow \partial B_{r,v})\le Ce^{ - m/C}\quad \mbox{or}\quad \max_{v} \pi_{\Lambda_m^\eta\cap B_{r,v}^\zero}(\partial B_{r/2,v}\stackrel{\dis}\longleftrightarrow\partial B_{r,v}) \le Ce^{ - m/C}\,.
   \end{align}
    Then the random-cluster model on satisfies SSM on $\Lambda_n^\eta$. 
\end{lemma}

\begin{proof} Fix any $v\in \Lambda_n$, and for simplicity let $B_{r}^\one$ denote $B_{r,v}$ with boundary conditions that are induced by $\eta$ on $\partial \Lambda_n$ and $\one$ on $\partial B_{r,v}\setminus \partial \Lambda_n$. The proof goes by constructing a coupling between $\omega^\zero \sim \pi_{B_r^\zero}$ and $\omega^\one\sim \pi_{B_r^\one}$ such that they agree on $B_{r/2}$ on the complement of the event 
    \begin{align}\label{eq:connectivity-event-ssm}
        \{\omega^\one: \partial B_{r/2}\stackrel{\ord}\longleftrightarrow \partial B_r\setminus \partial \Lambda_n\}. 
    \end{align}
    This coupling goes as follows. Expose  the connected component of $\partial B_{r}\setminus \partial \Lambda_n$ under $\omega^\one$ and call that set~$D$.  To relate this to the continuum construction of Section~\ref{sec:q-cluster-expansion}, this is the set of vertices in the connected component of $\mathbf{B}_r\setminus \Gamma(\omega^\one)$ that is incident to $\partial B_r \setminus \partial \Lambda_n$. Since the set~$D$ is bounded by a contour in $\Gamma(\omega^\one)$, the set of edges $\partial_\out D$ are revealed to be closed under $\omega^\one$, and together with $\partial \Lambda_n$ they enclose the region $B_r\setminus D$. Also, under the grand monotone coupling, expose the configuration $\omega^\zero(D\cup \partial_\out D)$, and note that by monotonicity all edges of $\omega^\zero(\partial_\out D)$ will also be closed. 
    
    We next sample $$\omega^\one(E(B_r \setminus D))\sim \pi_{B_r^\one}(\cdot \mid \omega^\one(D\cup \partial_\out D))\,,$$ 
    which is exactly the random-cluster distribution on $E(B_r\setminus D)$ with  boundary conditions induced by free on $\partial_\out D$ and $\eta$ on the portions of $\partial E(B_r\setminus D)$ that intersect $\Lambda_n$. In particular we claim that, because $\eta$ is side-homogeneous, the induced boundary conditions on $B_r\setminus D$ are independent of the values of $\omega^\one(D)$. This follows because the vertices of $\partial \Lambda_n$ that are incident to $B_r\setminus (D\cup \partial_\out D)$ are either on a wired side of~$\eta$, in which case they are wired regardless, or they are on a free side of~$\eta$, in which case they cannot be wired up to any other vertex through $\omega^\one(D)$ because they are not incident to~$D$. 
    
    The coupling can now be concluded by setting $\omega^\zero(E(B_r\setminus D)= \omega^\one(E(B_r\setminus D))$. Under the above coupling, $\omega^\one(B_{r/2}) = \omega^\zero(B_{r/2})$ as long as $B_{r/2}\cap D\ne \emptyset$. This means that for any $v$, 
    \begin{align*}
        \|\pi_{B_{r}^\one}(\omega(B_{r/2}\in \cdot) -\pi_{B_{r}^\zero}(\omega(B_{r/2}\in \cdot) \|_\tv \le \pi_{B_r^\one}(B_{r/2}\cap D \ne \emptyset) = \pi_{B_r^\one}(\partial B_{r/2}\stackrel{ord}\longleftrightarrow \partial B_r)\,.
    \end{align*}
    Maximizing both sides over $v$, this reduces SSM on $\Lambda_n^\eta$ to 
    the first inequality in~\eqref{eq:exp-decay-ssm}.
    
    The proof assuming the bound on the disordered connectivity events in~\eqref{eq:exp-decay-ssm} proceeds analogously, by revealing the $\dis$ components using the dual notion of connectivity in $F_*(\omega)$. All vertices adjacent to the revealed set $D$ will be wired to one another, and the same claim about the induced boundary conditions on $E(B_r\setminus D)$ being independent of the edge values on $D$ holds. 
\end{proof}

\begin{proof}[\textbf{\emph{Proof of Corollary~\ref{cor:intro2}}}]
    Let us start with item (i). By Theorem~\ref{thm:intro1} it suffices to establish that SSM holds on $\Lambda_n^\zero$ whenever $p<p_c(q,d)$. We wish to boost exponential decay of connectivities in the bulk, in the form of the first bound in~\eqref{eq:exp-decay-wsm}, to uniform exponential decay of connectivities in the presence of \emph{free} boundary conditions in the form of the first bound in~\eqref{eq:exp-decay-ssm} with $\eta = \zero$. In order to see this, note that by monotonicity
\begin{align*}
\pi_{\Lambda_m^\zero \cap B_{r,v}^\one}(\partial B_{r/2,v}\stackrel{\ord}\longleftrightarrow \partial B_{r,v}) \le \pi_{\Lambda_r^\one}(\partial \Lambda_{r/2}\stackrel{\ord}\longleftrightarrow \partial \Lambda_r)
\end{align*}
for every $v\in \Lambda_m$. For any $d\ge 2$, $q\ge 1$, the latter decays exponentially whenever $p<p_c(q,d)$ by~\cite{DCRT19}, so SSM on $\Lambda_n^\zero$ holds whenever $p<p_c(q,d)$. 
    
    We turn now to item (ii). At $p$ sufficiently large, we can compare the set of closed edges to a sub-critical percolation process on the graph induced by the dual notion of adjacency we considered for closed edges. Since this is a comparison to an independent percolation process, it is independent of the boundary condition~$\eta$, and therefore implies that the second bound in~\eqref{eq:exp-decay-ssm} holds, which by Lemma~\ref{lem:exp-decay-to-SSM} in turn implies SSM for $\Lambda_n^\eta$ for all side-homogeneous $\eta$. 
\end{proof}

\begin{remark}\label{rem:cylinders}
For $0\le k \le d$, define the free cylinder $(\mathbb Z/n\mathbb Z)^k \times \Lambda_n^\zero$ for $\Lambda_n$ in $d-k$ dimensions, and similarly the wired cylinder $(\mathbb Z/n\mathbb Z)^k \times \Lambda_n^\one$. As hinted at in Remark~\ref{rem:intro-cylinders}, 
the same argument used to prove Corollary~\ref{cor:intro2} could also be applied to deduce fast mixing on the free cylinder when $p<p_c(q,d)$ and the wired cylinder when $p>p_c(q,d)$. To see this generalization, in the above proof of Corollary~\ref{cor:intro2}, one simply takes the balls $B_{r,v}$ and $B_{r/2,v}$ to be on the cylinder (i.e., wrapping around in its first $k$ coordinates): with this change, the rest of the argument goes the same way.
\end{remark}

\subsection{Weak spatial mixing within a phase}
We conclude the section by using the results of Section~\ref{sec:q-cluster-expansion} to establish WSM within a phase for the random-cluster model at large $q$. This also essentially goes via a reduction to exponential tails on connectivities, but in this case under $\pi_\ord$ and $\pi_\dis$ as established in Lemma~\ref{lem:conditional-exp-decay-contours}. The fact that these are conditional measures introduces some complications that are handled with a somewhat more involved coupling argument.

\begin{lemma}\label{lem:large-q-wsm-within-phase}
    Fix $d\ge 2$ and suppose $q\ge q_0(d)$. There is a constant $C(q,d)>0$ such that the random-cluster model on $\bbT_n$ satisfies 
    \begin{itemize}
        \item WSM \emph{within the free phase} with constant $C$ uniformly over $p\le p_c(q,d)+o(n^{-1})$; and
        \item WSM \emph{within the wired phase} with constant $C$ uniformly over $p\ge p_c(q,d) - o(n^{-1})$.
    \end{itemize}
\end{lemma}

We will prove the result for the free phase, the proof for the wired phase being analogous. The proof goes by constructing a coupling between $\omega^\zero \sim  \pi_{\Lambda_r^\zero}$ and $\widecheck \omega \sim \widecheck \pi_n$ such that they agree on $\Lambda_{r/2}$ except with probability $Ce^{ - r/C}$. (Here, we are thinking of $\Lambda_r$ as being embedded in $\bbT_n$ by identifying the vertices of $\bbT_n$ with those of~$\Lambda_n$.) By virtue of Lemma~\ref{lem:equivalence-of-ord-and-hat}, it actually suffices to construct the coupling with $\omega_\dis \sim \pi_\dis$ instead. Let $\mathcal E_\dis$ be the set of configurations $\omega$ on $E(\Lambda_r)^c = E(\bbT_n)\setminus E(\Lambda_r)$ such that $\omega\cup E(\Lambda_r)\in \Omega_\dis$ (i.e., even if all edges in $E(\Lambda_r)$ are present, it still forms a configuration in $\Omega_\dis$). This event is important because it allows us to drop the conditioning in $\pi_\dis$ and apply properties of the unconditional random-cluster measure such as  monotonicity and domain Markov. 

\begin{lemma}\label{lem:conditioning-dropped}
    Suppose $\eta \in \mathcal E_\dis$. Then 
    \begin{align*}
        \pi_\dis(\omega(E(\Lambda_r))\in \cdot \mid \omega(E(\Lambda_r)^c) = \eta) = \pi_n(\omega(E(\Lambda_r))\in \cdot \mid \omega(E(\Lambda_r)^c) = \eta) = \pi_{\Lambda_r^\eta}(\omega\in \cdot)\,.
    \end{align*}
\end{lemma}
\begin{proof}
    The second inequality is simply the domain Markov property. Now consider any event $A$ on configurations on $E(\Lambda_r)$. Then, 
    \begin{align*}
        \pi_\dis(A \mid \omega(E(\Lambda_r)^c) = \eta)   = \frac{\pi_n(\omega(E(\Lambda_r))\in A, \omega(E(\Lambda_r)^c) = \eta, \Omega_\dis)}{\pi_n(\omega(E(\Lambda_r)^c) = \eta, \Omega_\dis)}\,.
    \end{align*}
    But the configuration $\eta$ (together with empty on $E(\Lambda_r)$) is in $\Omega_\dis$, and $\Omega_\dis$ is a decreasing set, so $\eta \cup \omega(E(\Lambda_r))$ is always in $\Omega_\dis$. As such, the intersection with $\Omega_\dis$ can be dropped from both the numerator and denominator, giving the desired first equality.  
\end{proof}

The main coupling of this section goes as follows. We describe it in the free phase, with the construction for the ordered phase being symmetrical.

\begin{definition}\label{def:wsm-within-phase-coupling}
The coupling $\mathbb P$ of $(\omega^\zero,\omega_\dis)$ is constructed as follows: 
\begin{enumerate}
    \item Sample the configuration $\omega_\dis(E(\Lambda_r)^c) \sim \pi_{\dis}$ where $E(\Lambda_r)^c = E(\bbT_n)\setminus E(\Lambda_r)$
    \item If $\omega_\dis(E(\Lambda_r)^c)\notin \mathcal E_\dis$, sample independently 
    $$\omega^\zero \sim \pi_{\Lambda_r^\zero}\qquad \mbox{and}\qquad \omega_\dis \sim \pi_\dis(\cdot \mid \omega_\dis(E(\Lambda_r)^c))\,.$$
    \item If $\omega_\dis(E(\Lambda_r)^c) \in \mathcal E_\dis$, then independently from $\partial \Lambda_r$ inwards, reveal all connected components of $\partial \Lambda_r$ under 
    $$\omega_\dis \sim \pi_n(\omega(E(\Lambda_r))\in \cdot\mid \omega(E(\Lambda_r)^c))\,,$$
    and call that set $D$. These will evidently be the interiors of all contours labeled $\ord$ that intersect $\partial \Lambda_r$. Thus we will also have revealed their boundary $\partial_\out D$ to be all-closed in $\omega_\dis$, and the boundary conditions induced by the set of revealed edges on $\Lambda_r\setminus D$ will be all-free. Under the monotone coupling of $\omega^\zero$ to $\omega_\dis$ we can also reveal all edges of $\omega^\zero(D\cup \partial_\out D)$ and they will also be all closed on $\partial_\out D$, therefore also inducing free boundary conditions on $\Lambda_r \setminus D$. 
    \item Sample $\omega_\dis(E(\Lambda_r\setminus D))$ according to the random-cluster model on that set with free boundary conditions, and set $\omega^\zero(E(\Lambda_r\setminus D)) = \omega_\dis(E(\Lambda_r\setminus D))$. 
\end{enumerate}
\end{definition}

The validity of this coupling can be seen quite easily with the key point being the use of Lemma~\ref{lem:conditioning-dropped} to ensure that in items (3)--(4), $\omega_\dis$ stochastically dominates $\omega^\zero$ and satisfies the domain Markov property. 

\begin{proof}[\textbf{\emph{Proof of Lemma~\ref{lem:large-q-wsm-within-phase}}}]
    Let us begin with the proof for the free phase. As mentioned, by Lemma~\ref{lem:equivalence-of-ord-and-hat}, it suffices for us to establish that 
    \begin{align*}
        \|\pi_{\Lambda_{r}^\zero}(\omega(\Lambda_{r/2})\in \cdot) - \pi_\dis(\omega(\Lambda_{r/2})\in \cdot)\|_\tv \le Ce^{ - r/C}\,.
    \end{align*}
    Under the coupling of Definition~\ref{def:wsm-within-phase-coupling}, it is evident that 
    \begin{align}\label{eq:2-terms-wsm-within-phase-coupling}
        \mathbb P(\omega^\zero(\Lambda_{r/2})\ne \omega_\dis(\Lambda_{r/2}))\le \pi_\dis(\omega(E(\Lambda_r)^c)\notin \mathcal E_\dis) + \pi_\dis(\partial \Lambda_r\stackrel{\ord}\longleftrightarrow \partial\Lambda_{r/2})\,.
    \end{align}
    Let us start with the first of these events. In order for a configuration to not be in $\mathcal E_\dis$, it must be the case that there is a contour in $\omega_\dis$ that connects opposite sides of $\partial \Lambda_r$ in $\mathbf{T}_n \setminus \mathbf{\Lambda}_r$. This is because there must be a connected component $\mathcal C$ of $\partial \mathbf{\omega}_\dis(\bbT_n \setminus \Lambda_r)$, such that if $\bgamma_\ord$ is the contour of $\mathbf{\omega}_\dis \cup \mathbf{\Lambda}_r$ and $\bgamma_\dis$ is the contour of $\mathbf{\omega}_\dis \setminus \mathbf{\Lambda}_r$ containing $\mathcal C$, then either: 
    \begin{itemize}
        \item the topological triviality/non-triviality of $\bgamma_\dis$ and $\bgamma_\ord$ are different; or
        \item the topological triviality/non-triviality of the connected component of $\mathbf{T}_n\setminus \bgamma_\dis$ and $\mathbf{T}_n\setminus \bgamma_\ord$ containing the open edges of $\omega_\dis$ are different. 
    \end{itemize}
    The first of these holds if $\omega_\dis \cup E(\Lambda_r)\in \Omega_{\rm tunnel}$, and the second holds if $\omega_\dis \cup E(\Lambda_r)\in \Omega_\ord$. Either of these events necessitates that, in some direction of the torus, there is both an $\ord$ and a $\dis$ connection connecting two opposite sides of $\partial \Lambda_r$, which implies in particular that $\mathcal C$ has length at least $n-r$. Therefore, by a union bound and Lemma~\ref{lem:conditional-exp-decay-contours}, 
    \begin{align*}
        \pi_\dis(\omega(E(\Lambda_r)^c)\notin \mathcal E_\dis) \le \pi_\dis(\exists v\in \partial\Lambda_r: \|\bgamma_v\|\ge n-r) \le Ce^{ - c\beta (n-r)}
    \end{align*}
    uniformly over all $p\le p_c(q,d) + o(n^{-1})$. Since in the definition of WSM within a phase we only consider $r\le n/2$, this is exponentially decaying in~$r$. 
    
    The second term of~\eqref{eq:2-terms-wsm-within-phase-coupling} is also bounded by a union bound together with Lemma~\ref{lem:conditional-exp-decay-contours} via 
    \begin{align*}
        \pi_\dis(\partial \Lambda_r\stackrel{\ord}\longleftrightarrow \partial\Lambda_{r/2}) \le \pi_\dis(\exists v\in \partial \Lambda_r: \|\bgamma_v\|\ge r/2)\le Ce^{ - c\beta r}\,.
    \end{align*}
    Putting these two bounds together, we obtain the desired property of WSM within the free phase. 
    
    To establish the analogous bound for the wired phase, the event~$\mathcal E_\ord$ and the coupling of Definition~\ref{def:wsm-within-phase-coupling} are defined symmetrically, and the remainder of the proof proceeds as before. 
\end{proof}

\section{FK dynamics from random phase initializations near criticality}\label{sec:learning-weights}
In this section, we use an interpolation argument together, crucially, with the uniformity of certain of our WSM within a phase estimates through a microscopic window around the critical point, to devise a polynomial time algorithm for learning the relative weights of the $\one$ and $\zero$ phases in order to initialize the random phase initialization appropriately. We first state a more general version of Theorem~\ref{thm:intro3}, showing fast mixing within a phase on either side of the critical point. 

\subsection{Uniform mixing time bounds from random phase initializations}
We first combine the mixing time estimates within a phase from Section~\ref{sec:mixing-within-a-phase} with the parameter regimes in which WSM within each of the wired and free phases hold at $q$ large to obtain the following result on mixing times of the FK dynamics restricted to the wired and free phases respectively. We draw attention to the fact that the mixing time bound is uniform over all~$p$ even microscopically {\it beyond\/} the critical point in each direction.

\begin{theorem}\label{thm:uniform-fast-mixing}
For every $d\ge 2$, there exists $q_0(d)$ and $C(q,d)$ such that, for all $q\ge q_0$, the following holds uniformly over sequences $p_n \ge p_c - o(1/n)$: if 
\begin{align}\label{eq:T-n}
    t_*^n = \exp(C(\log n)^{d-1})\,,
\end{align}
then the restricted FK dynamics on $\bbT_n$ has 
\begin{align*}
    \|\mathbb P(\widehat X_{t_*^n}^\one\in \cdot) - \widehat \pi_n \|_\tv = O(n^{-100})\,. 
\end{align*}
The analogous claim holds uniformly over $p_n \le p_c + o(1/n)$ with the free phase instead of the wired phase. 
\end{theorem}

By Lemma~\ref{lem:unstable-phase-negligible} (together with Lemma~\ref{lem:tunnel-doesnt-matter}), when $|p_n-p_c| = \Omega(1)$ the restricted FK
dynamics in Theorem~\ref{thm:uniform-fast-mixing} in fact produces a sample that is within total variation distance $n^{-100}$ of $\pi_n$ itself. When $p_n$ is sufficiently close to $p_c$, however, both the wired and free phases have non-negligible mass under~$\pi_n$, so that neither of the restricted chains is sufficient to obtain a sample from~$\pi_n$. The following corollary shows that, close to the critical point, we can combine the restricted dynamics in each of the two phases to show that the {\it unrestricted\/} FK dynamics initialized from an appropriate mixture of the all-free and all-wired configurations also mixes quickly and produces samples close to~$\pi_n$.

\begin{corollary}\label{cor:random-phase-initialization-general-weight}
    Fix $d\ge 2$, $q\ge q_0(d)$, and a sequence $|p_n -  p_c| - o(1/n)$. Suppose that the sequence $m_n^*$ satisfies $|m_n^* - \pi_n(\widehat \Omega)|\le \eta_n$ for some sequence $\eta_n=o(1)$. Then the FK dynamics on $\bbT_n$ initialized from 
    \begin{align*}
        \nu_*^{\zero/\one} := (1-m_n^*) \delta_\zero + m_n^* \delta_\one
    \end{align*}
    satisfies the following: 
    \begin{align*}
       \| \mathbb P (X_{t_*^n} \in \cdot) - \pi_n\|_\tv \le O(n^{-100}) + 4 \eta_n\,.
    \end{align*}
\end{corollary}

\begin{proof}[\textbf{\emph{Proof of Theorem~\ref{thm:uniform-fast-mixing}}}]
    Our aim is to apply Theorem~\ref{thm:intro3}.  First, by Lemma~\ref{lem:exponential-stability-proof}, $\widehat \Omega$ is exponentially stable with constant $C_1$ uniformly over all $p\ge p_c- o(1/n)$, and by Lemma~\ref{lem:large-q-wsm-within-phase} WSM within the wired phase holds with constant $C_2$ uniformly over $p\ge p_c-o(1/n)$. Thus, by Theorem~\ref{thm:intro3}, there exist $C_0,K_0$ depending only on $C_1, C_2$ such that if $g_n(t)$ is as in~\eqref{eq:g(t)} for $K=K_0$, then for every $t\ge 0$, 
    \begin{align*}
        \|\mathbb P(\widehat X_{t,\bbT_n}^\one\in \cdot ) - \widehat \pi_n\|_\tv \le C_0 n^d \exp(- g_n(t)/C_0)\,.
    \end{align*}
    To plug in a value for $g_n(t)$, we recall that $f(m)$ in its definition is given by~\eqref{eq:assumption-scale-m}, and as described there, one can plug in for $f(m)$ a worst-case mixing time bound for the dynamics on~$\Lambda_m^\one$. 
    
    By a standard canonical paths argument (originating in~\cite{JS}; see e.g.,~\cite[Proposition 4.2]{GL3} for a formulation for the random-cluster model with free or wired boundary conditions), the (worst-case) mixing time of the FK dynamics on $\Lambda_m^{\one}$ is at most $\exp(C_3 m^{d-1})$ for some $C_3(q)>0$, uniformly over all $p\ge p_c(q)/2$, say, so that certainly $p\ge p_c(q)- o(1/m)$. In particular, \eqref{eq:assumption-scale-m}~is satisfied by the choice $f(m) = C_3 e^{C_3 m^{d-1}}$. In turn, $g_n(t)$ of~\eqref{eq:g(t)} satisfies $g_n(t)\ge C_4^{-1}(\log t)^{1/(d-1)}\wedge n$ for some other constant~$C_4$, again uniformly over $p\ge p_c - o(1/n)$. At this point, the desired result follows from a direct application of Theorem~\ref{thm:intro3}, since $g_n(t_*^n)$ is seen to be smaller than $n^{-100}$ if the constant $C$ in $t_*^n$ exceeds $100 C_0 C_4 d^2$. 
    
    The bound for the free phase is obtained by a symmetrical argument.
\end{proof}

\begin{proof}[\textbf{\emph{Proof of Corollary~\ref{cor:random-phase-initialization-general-weight}}}]
    By a triangle inequality, we can control the total variation distance as
\begin{align}\label{eq:vdist}
    \|\mathbb P(X_{t}^{\nu_*^{\zero/\one}} & \in \cdot)  - \pi_n\|_\tv \nonumber \\
    & \le \big\|m_n^* \mathbb P(X_{t}^{\one}\in \cdot ) + (1-m_n^*)\mathbb P(X_{t}^{\zero}\in \cdot) - m_n^* \widehat \pi_n - (1-m_n^*) \widecheck\pi_n \big\|_\tv  + 2|m_n^* - \pi_n(\widehat \Omega)|\,. 
\end{align}
By the assumption on $m_n^*$, the second term satisfies $2|m_n^* - \pi_n(\widehat \Omega)| \le 2\eta_n$. 
By another triangle inequality, the first term in~\eqref{eq:vdist} is at most
\begin{align*}
        m_n^* \|\mathbb P(X_{t}^{\one}\in \cdot )- \widehat \pi_n\|_\tv + (1-m_n^*)\|\mathbb P(X_{t}^{\zero}\in \cdot ) - \widecheck \pi_n\|_\tv\,.
\end{align*}
The two terms on the right are handled analogously, so we only consider one of them. By another triangle inequality, we have
\begin{align}\label{eqn:ajs4}
    \|\mathbb P(X_{t}^{\one}\in \cdot) - \widehat \pi_n\|_\tv \le \|\mathbb P(X_{t}^{\one}\in \cdot) - \mathbb P(\widehat X_{t}^{\one}\in \cdot)\|_\tv +  \|\mathbb P(\widehat X_{t}^{\one}\in \cdot) - \widehat \pi_n\|_\tv\,. 
\end{align}
By~\eqref{eq:monotonicity-relations} and the grand coupling, the first term in~\eqref{eqn:ajs4} is at most $\mathbb P(\widehat\tau^{\one}\le t) \le C_3 e^{ - n^{d-1}/C_3}$ while $t\le e^{n^{d-1}/K}$. The second term is exactly the term shown in Theorem~\ref{thm:uniform-fast-mixing} to be $O(n^{-100})$ when $t= t_*^n$. 
\end{proof}

\begin{proof}[\textbf{\emph{Proof of Theorem~\ref{thm:intro4}}}]
Theorem~\ref{thm:intro4} is a special case of Corollary~\ref{cor:random-phase-initialization-general-weight} exactly at $p=p_c$, where we know from~\cite[Lemma 6.1]{BCT} and Lemma~\ref{lem:equivalence-of-ord-and-hat} that if $m_n^* = q/(q+1)$ then the assumption on $m_n^*$ holds with $\eta_n = \exp( - \Omega(n))$. 

In order to get the improved bound $t_*^n= n^{o(1)}$ when $d=2$, we replace the crude canonical paths bound used on the local mixing time quantity $f(m)$ with the $\exp(n^{o(1)})$ bound of~\cite{GL1} on the mixing times on $\Lambda_n^\one$ and $\Lambda_n^\zero$ at $p=p_c$. With that choice, $g_n(t)$ from~\eqref{eq:g(t)} satisfies $g_n(t)\ge C_4^{-1} (\log t)^{\omega(1)} \wedge n$, so that $t = \exp((\log n)^{o(1)}) = n^{o(1)}$ suffices to attain a small total variation distance to stationarity. (The extra factor of $N$ in the mixing time is of course due to the change from discrete to continuous time.)
\end{proof}

\subsection{Efficiently approximating the weights for the random phase initialization}\label{subsec:weights}
Given a value of $p$, in order to obtain a good sample from $\pi_n$ on $\bbT_n$ for large $n$ using Theorem~\ref{thm:uniform-fast-mixing} and Corollary~\ref{cor:random-phase-initialization-general-weight}, it is imperative to have a good approximation~$m_n^*$, i.e. deduce, for a fixed $p$ and size~$n$, the relative weights of $\widehat \Omega$ and $\widecheck\Omega$. 
In particular, at finite~$n$, it could be that $p$ differs only microscopically from~$p_c$ but the correct weighting for the initial distribution in Corollary~\ref{cor:random-phase-initialization-general-weight} is not $(\frac{1}{q+1},\frac{q}{q+1})$ but rather some other non-negligible pair of weights. 

The following lemma describes an MCMC algorithm that is (up to polynomial factors) as efficient as the mixing times in Corollary~\ref{cor:random-phase-initialization-general-weight} itself, which produces a
weight~$m_n^*$ that is a good approximation to $\pi_n(\widehat \Omega)$ whenever the weight is non-trivial.  Our approach is to approximate the partition functions $\widehat Z_p,\widecheck Z_p$ of $\widehat \Omega$ and $\widecheck \Omega$, respectively, at parameter~$p$ by interpolating from $p=1$ down to $p_c$ with FK dynamics chains restricted to the wired phase, and up from $p=0$ to $p_c$ with FK dynamics chains restricted to the free phase.

We claim that it suffices to approximate $\pi_n(\widehat \Omega)$ (or, symmetrically, $\pi_n(\widecheck \Omega)$) when $|p- p_c| = o(1/n)$. This is because, by monotonicity, $\pi_n(\widecheck\Omega)$ decreases as $p$ grows, and by Lemma~\ref{lem:unstable-phase-negligible} there will be some $p$ within $o(1)$ of $p_c$ at which $\widecheck \Omega$ is exponentially negligible (and symmetrically for $\pi_n(\widehat\Omega)$).

\begin{lemma}\label{lem:learning-weights}
Let $d\ge 2$, $q \ge q_0(d)$, and consider any sequence $p=p_n$ such that $|p- p_c|= o(1/n)$.  
There exists an MCMC-based algorithm that, using $\mbox{poly}(n)$ samples from restricted FK dynamics run for time $t_*^n$, outputs an
estimate of $m_n^*$ satisfying
\begin{align*}
    \mathbb P\big(|m_n^* -  \pi_n(\widehat \Omega)|>n^{-100}\big) =  O(n^{-100})\,.
\end{align*}
\end{lemma}

\begin{proof}
Observe that, by definition,  
\begin{align}\label{eq:pidefn}
    \pi_n(\widehat \Omega) = \frac{\widehat Z_p}{\widehat Z_p + \widecheck Z_p}\,.
\end{align}
Our aim is to approximate each of $Z_p$ and $Z_p$ within $1\pm O(n^{-50})$ multiplicative factors; combining these estimates as in~\eqref{eq:pidefn} yields an
estimate of $\pi_n(\widehat \Omega)$ within the same ratio. We will explain how to approximate $\widecheck Z$ using an increasing sequence of values of~$p$ starting at $p=0$; the
approximation scheme for $\widehat Z$ works symmetrically using a decreasing sequence 
starting at $p=1$.

For any $p: |p-p_c|  = o(1/n)$, construct an increasing sequence $(p^{(i)})_{i=0,...,K}$ such that $p^{(0)}=0$, $p^{(K)} =p$, and $p^{(i)} - p^{(i-1)}\le n^{-d}$ for all $i$.     
Then, we can expand $\widecheck Z_{p^{(K)}}$ as 
\begin{align}\label{eq:Z-dis-cooling-schedule}
    \widecheck Z_{p^{(K)}} = \widecheck Z_{p^{(0)}} \prod_{1\le i\le K} \frac{\widecheck Z_{p^{(i)}}}{\widecheck Z_{p^{(i-1)}}}\,.
\end{align}
It is now easily seen that each ratio in the product appearing above is actually the expectation of a bounded random variable under the Gibbs measure at parameter $p^{(i-1)}$. More precisely, letting $W_p(\omega) = p^{|\omega|}(1-p)^{|E(\bbT_n)|- |\omega|}q^{\Comp(\omega)}$, we have 
\begin{align*}
   \frac{\widecheck Z_{p^{(i)}}}{\widecheck Z_{p^{(i-1)}}} = \widecheck \pi_{n,p^{(i-1)}}\Big[\frac{W_{p^{(i)}}(\omega)}{W_{p^{(i-1)}}(\omega)}\Big] \le \max_{\omega}\Big(\frac{p^{(i)}(1-p^{(i-1)})}{p^{(i-1)}(1-p^{(i)})}\Big)^{|\omega|}\,.
\end{align*}
Using the fact that $p^{(i)} - p^{(i-1)} \le n^{-d}$ and $|\omega|\le |E(\bbT_n)|  = O(n^{d})$, the above ratio is bounded uniformly by some constant, say $A$. As a consequence, if for each $i$, $\widecheck X_t^{(i)}$ is the final configuration after $t$ steps of the restricted (to $\widehat \Omega$) FK dynamics initialized from $\zero$ on the torus $\bbT_n$ at parameter value $p^{(i)}$, then we can estimate the values of the ratios in~\eqref{eq:Z-dis-cooling-schedule} via
\begin{align}\label{eq:expectn}
    \Big|\mathbb E\Big[\frac{{W}_{p^{(i)}}(\widecheck X^{(i-1)}_{t})}{W_{p^{(i-1)}}(\widecheck X^{(i-1)}_{t})}\Big] -\widecheck \pi_{n,p^{(i-1)}}\Big[\frac{W_{p^{(i)}}(\omega)}{W_{p^{(i-1)}}(\omega)}\Big]\Big| \le A
    \Big\|\mathbb P(\widecheck X^{(i-1)}_{t}\in \cdot) - \widecheck \pi_{n,p^{(i-1)}}\Big\|_\tv\,.
\end{align}
By Theorem~\ref{thm:uniform-fast-mixing}, since $p^{(i)}\le p_c + o(1/n)$ for all $i$, the total variation distance on the right is at most $n^{-100}$ if $t \ge t_*^n$ from~\eqref{eq:T-n}. 

Notice further that, algorithmically, we can approximate the expectation on the left-hand side of~\eqref{eq:expectn} using $\ell$ independent runs of the restricted FK dynamics initialized from the $\zero$ configuration, up to an error of $\epsilon_\ell = \ell^{\delta - \frac 12}$, except with probability $\exp(-\Omega(\ell^{2\delta}))$. (Here, we are using the uniform boundedness of the ratio of weights to apply standard concentration estimates.) Taking $\ell$ to be a large polynomial in $n$, this error can be absorbed.  

Altogether, we can suppose we have a sequence of estimates $a_i$ of the above expectation, so that for all $i$, 
\begin{align*}
    \Big|a_{i-1} - \widecheck \pi_{n,p^{(i-1)}}\Big[\frac{W_{p^{(i)}}(\omega)}{W_{p^{(i-1)}}(\omega)}\Big]\Big| = O(n^{-100})\,.
    \end{align*}
    By boundedness of the expectation under consideration, we then have 
    \begin{align*}a_{i-1} = (1+O(n^{-100}))\widecheck \pi_{n,p^{(i-1)}}\Big[\frac{W_{p^{(i)}}(\omega)}{W_{p^{(i-1)}}(\omega)}\Big]\,.
\end{align*}
In that case, by~\eqref{eq:Z-dis-cooling-schedule}, since $K\le n^d$,
\begin{align*}
    Z_{p^{(0)}} \prod_{1\le i\le K} a_i  =  Z_{p^{(0)}}\prod_{1\le i\le K} \Big(\frac{\widecheck Z_{p^{(i)}}}{\widecheck Z_{p^{(i-1)}}}\Big)(1+O(n^{-100})) \le \widecheck Z_{p^{(K)}} (1+O(n^{-100}))\,.
\end{align*}
Furthermore, $Z_{p^{(0)}} = Z_{0}$ is of course easy to compute exactly (it is just the weight of a single, empty configuration). Thus, $\widecheck Z_{p^{(K)}}$ can be approximated to multiplicative factor $1\pm O(n^{-100})$ by a Markov chain approach consisting of polynomially many independent runs at $n^d$ many parameter values~$p$, each of which have run-time bounded by $t_*^n = \exp(C(\log n)^{d-1})$. Reasoning analogously from the other direction, we can also estimate $\widehat Z_{p^{(K)}}$ with a matching running time, and combine these two estimates to obtain an estimate at~$m_n^*$ with the desired accuracy. 
\end{proof}

\subsection*{Acknowledgements}
The authors thank the anonymous referees for their careful readings and helpful suggestions. R.G.\ acknowledges the support of NSF DMS-2246780 and the Miller Institute for Basic Research in Science. The research of A.S.\ is supported in part by NSF grant CCF-1815328.  Part of this work was done while A.S.\ was visiting EPFL, Switzerland.

\appendix

\section{Solution to the recurrence of Lemma~\ref{lem:recurrence-solution}}\label{sec:recurrence-solution}
In this section, we establish the bound claimed in  Lemma~\ref{lem:recurrence-solution}. The argument is standard---see e.g.~the more general bound of~\cite[Lemma 17]{HarelSpinka}---but we include a short proof for self-containedness.

\begin{proof}[\textbf{\emph{Proof of Lemma~\ref{lem:recurrence-solution}}}]
We prove the bound in two steps, first establishing that the sequence decays at least as a stretched exponential, then boosting this to the desired exponential decay. First of all, take $r= r(k) := - C_0 \log a_k$ for $C_0$ sufficiently large to be specified later. Let 
$$\kappa_\delta := \inf \{k: a_k\le 2^{- \delta k}\} \wedge \frac{n}{2}\,.$$
If $\delta$ is sufficiently small as a function of $C_0$, then while $k\le \kappa_\delta$ we have $r(k)\le k$.  
Plugging this choice of $r$ into the recurrence relation, for all $k\le \kappa_\delta$, 
\begin{align*}
    a_{2k} \le d(2C_0 \log a_k^{-1})^d a_k^2 + a_k^{C_0/C_\star}\,.
\end{align*}
There evidently exist $\epsilon_0 \in (0,1/2)$ and $C_0>0$ depending only on $C_\star$ and $d$, so that for all $k_0 \le k\le \kappa_\delta$, 
\begin{align*}
    a_{2k}\le a_{k}^{1.99}\,.
\end{align*}
Thus, for all $l\le \log_{1.99}(\kappa_\delta /k_0)$, we get 
\begin{align*}
    a_{2^l k_0} \le a_{k_0}^{1.99^l}
\end{align*}
Since $a_{k_0}\le 1/2$, we deduce that for all $k_0 \le k\le \kappa_\delta$ 
\begin{align}\label{eq:stretched-exponential-decay}
    a_{k} \le \Big(\frac{1}{2}\Big)^{(k/k_0)^{.99}}
\end{align}
using the inequality $1.99^l \ge 2^{.99 l}$. 

We now boost this to an exponential decay. 
Take $r= r(k) := k$, so that  
\begin{align*}
    a_{2k} \le d(2k)^d a_k^2 + e^{- k/C_\star}\qquad \mbox{for all $k\le n/2$}\,.
\end{align*}
Now let 
$$\psi_k := \sqrt{d}(4 k)^{d/2}\sqrt{a_k + e^{ - k/2C_\star}}\,.$$
Then, using $\sqrt{a+b} \le \sqrt a + \sqrt b$, we have for all $k\le \kappa_\delta$,
\begin{align*}
    \psi_{2k} & \le \sqrt{d} 4^{d/2}(2 k)^{d/2} \sqrt{d(2k)^d a_k^2 + 2 e^{-k/C_\star}} \\
    & \le d(4k)^d a_k  + \sqrt{d} 4^{d/2}(2k)^{d/2}\sqrt{2} e^{ - k/2C_\star}\,.
\end{align*}
In particular, this implies that for every $k\le \kappa_\delta$, 
\begin{align*}
    \psi_{2k}\le \psi_{k}^2\qquad \mbox{and}\qquad \psi_{2^\alpha k_0'}\le \psi_{k_0'}^{2^\alpha} \qquad \mbox{for every $k_0'$}
\end{align*}
and all $\alpha \le \log_{2}(\kappa_\delta/k_0')$. Let us find $k_0'$ such that $\psi_{k_0'}<1/2$, which is equivalent to asking for $k_0'$ such that 
\begin{align*}
    a_{k_0'} + e^{ - k_0'/2C_\star} < \frac12 (4k_0')^{-d}
\end{align*}
From the bound of~\eqref{eq:stretched-exponential-decay}, there would exist $k_0'(k_0,C_\star,d)$ such that such a bound holds. As a consequence, for all $k_0' \le k\le \kappa_\delta$, we would have $a_{k}\le \psi_{k} \le 2^{-k/k_0'}$
which together with the definition of $\kappa_\delta$, for $\delta$ sufficiently small depending only on $C_\star,d$ implies that for $k_0'\le k\le n/2$,
\begin{align*}
    a_k \le 2^{ - k/(\delta^{-1}\vee k_0')}
\end{align*}
The requirement that $k\ge k_0'$ can be absorbed by the pre-factor $C$, depending only on $k_0,C_\star,d$ in the claimed bound of Lemma~\ref{lem:recurrence-solution} concluding the proof.
\end{proof}

\section{Slow (worst-case) mixing of the FK dynamics near criticality}\label{app:slow-mixing}
In~\cite{BCT}, estimates of the form of Section~\ref{sec:q-cluster-expansion} were used to show that the Swendsen--Wang dynamics has $\exp(\Omega(n^{d-1}))$ (worst-case) mixing time on $\bbT_n$ at $\beta = \beta_c(q,d)$ and integer $q\ge q_0(d)$. Together with the comparison results of~\cite{Ullrich-random-cluster}, this implies that at integer $q\ge q_0(d)$, the mixing time of FK dynamics is $\exp(\Omega(n^{d-1}))$ at $p=p_c(q,d)$. For the sake of completeness, we demonstrate that one can similarly deduce $\exp(\Omega(n^{d-1}))$ mixing time at general (possibly non-integer) $q\ge q_0(d)$ and $p=p_c(q,d)$ using the same technology of Section~\ref{sec:q-cluster-expansion}. We also use the fact that those estimates apply in microscopic windows about $p_c$ to show that the exponential slowdown applies in a $o(1/n)$ window about~$p_c$. 

\begin{proposition}\label{prop:FK-slow-mixing-general-q}
    Fix $d\ge 2$, $q\ge q_0(d)$, and suppose $p= p_n$ is such that $|p-p_c| = o(1/n)$. Then the mixing time of the FK dynamics on $\bbT_n$ is $\exp(\Omega(n^{d-1}))$. 
\end{proposition}

\begin{proof}
    Recall (see, e.g.,~\cite{LP}) that for a Markov chain with transition matrix $P$ and stationary distribution $\pi$, 
    \begin{align*}
        \tmix \ge \frac{1}{2\Phi_\star}\,, \qquad \mbox{where}\qquad \Phi_\star = \inf_{A\subset \Omega} \Phi(A) = \inf_{A\subset \Omega} \frac{\sum_{\omega\in A,\omega'\notin A} P(\omega,\omega')\pi(\omega)}{\pi(A)\pi(A^c)}\,.
    \end{align*}
    Evidently, the numerator of $\Phi(A)$ is at most $\pi(\partial A)$. The slow mixing will be a consequence of the exponential stability of the wired and free phases $\widehat \Omega$ and $\widecheck \Omega$, which we recall from~\eqref{eq:def:hatOmega}--\eqref{eq:def:checkOmega} and Definition~\ref{def:exponentially-stable}. To be precise, in~\eqref{eq:def:hatOmega}--\eqref{eq:def:checkOmega}, fix $\epsilon(q,d)$ such that Lemma~\ref{lem:exponential-stability-proof} holds. 
    
    If $p\ge p_c$, then take $\widecheck \Omega$ as the bottleneck set $A$ and note that 
    \begin{align*}
        \tmix \ge \frac{\pi_n(\widecheck \Omega)\pi_n(\widehat\Omega)}{2\pi_n(\partial \widecheck \Omega)} = \frac{\pi_n(\widehat\Omega)}{2\widecheck \pi_n(\partial \widecheck \Omega)}\,.
    \end{align*}
    Since $p\ge p_c$, $\pi_n(\widehat \Omega) = \Omega(1)$ while $\widecheck \pi_n(\partial \Omega) = \exp(- \Omega(n^{d-1}))$ per Lemma~\ref{lem:exponential-stability-proof} since $p\le p_c+ o(1/n)$. 
    The argument when $p\le p_c$ proceeds symmetrically, taking $\widehat \Omega$ as the set $A$. 
    \end{proof}

\bibliographystyle{abbrv}
\bibliography{references}

\end{document}